\documentclass[twoside]{article}
\usepackage[accepted]{aistats2022}
\usepackage[round]{natbib}

\setcitestyle{authoryear, open={((},close={)}}
\usepackage{algorithm}
\usepackage{algpseudocode}
\usepackage{amsmath,amssymb,amsthm}
\usepackage{hyperref}
 \hypersetup{
         colorlinks   = true,
         linkcolor    = blue,
         citecolor    = blue
 }
 \usepackage{thmtools}
 \usepackage{thm-restate}
 \newtheorem{theorem}{Theorem}
 \newtheorem{proposition}{Proposition}
 \newtheorem{lemma}{Lemma}

 \newtheorem{assumption}{Assumption}
 \theoremstyle{remark}
 \newtheorem{remark}{Remark}

\DeclareMathOperator*{\argmin}{arg\,min}

\renewcommand{\t}{^{\top}}
\newcommand{\R}{\mathbb{R}}

\newcommand{\inv}{^{-1}}
\newcommand{\tr}{\text{\normalfont{Tr}}}

\newcommand{\eps}{\ensuremath{\varepsilon}}
\newcommand{\p}{\mathbb{P}}
\newcommand{\E}{\mathbb{E}}

\newcommand{\diag}{\text{diag}}

\newcommand{\sjp}{\Sigma_{JJ}^{\perp}}

\begin{document}
\runningtitle{Entrywise Sparse PCA}
\aistatstitle{Entrywise Recovery Guarantees for Sparse PCA \\ via Sparsistent Algorithms}
\aistatsauthor{ Joshua Agterberg \And Jeremias Sulam}
\aistatsaddress{Johns Hopkins University \And Johns Hopkins University} 

\begin{abstract}
Sparse Principal Component Analysis (PCA) is a prevalent tool across a plethora of subfields of applied statistics. While several results have characterized the recovery error of the principal eigenvectors, these are typically in spectral or Frobenius norms. In this paper, we provide entrywise $\ell_{2,\infty}$ bounds for Sparse PCA under a general high-dimensional subgaussian design.  In particular, our results hold for any algorithm that selects the correct support with high probability, those that are \emph{sparsistent}.  Our bound improves upon known results by providing a finer characterization of the estimation error, and our proof uses techniques recently developed for entrywise subspace perturbation theory.
\end{abstract}

\section{INTRODUCTION}
Principal component analysis (PCA) is a standard statistical technique for dimensionality reduction of data in an unsupervised manner.  Given i.i.d mean-zero observations $X_1,\dots, X_n \in \R^p$ with covariance matrix $\Sigma \in \R^{p\times p}$, the goal of PCA is to estimate the leading $k$-dimensional subspace of $\Sigma$, which has the interpretation of representing each observation as a linear combination of \emph{principal components}, where each principal component is a direction of maximal variance.  The classical theory of PCA (e.g. \citet{anderson_introduction_2003}) shows that if the number of covariates $p$ is fixed and the number of samples $n$ tends to infinity, then the leading eigenvectors of the sample covariance approximate the leading eigenvectors of the population covariance well.


 In the modern era of big data, it is often unrealistic to assume that $p$ remains fixed in $n$.
In the seminal work of \citet{johnstone_consistency_2009}, the authors introduced the \emph{spiked covariance model} where the leading eigenvalue of the population covariance satisfies $\lambda_1 > 1$, while all other eigenvalues are all 1.   In \citet{johnstone_consistency_2009}, the authors showed that if $\hat u_1$ is the leading eigenvector of the sample covariance  and $u_1$ is the leading eigenvector of the population covariance, then $\langle \hat u_1, u_1 \rangle$ need not tend to 1 as $p$ and $n$ tend to infinity unless either $p/n \to 0$ or the leading eigenvalue $\lambda_1$ tends to infinity.  They then went on to  show that if $\lambda_1$ remains bounded away from infinity but the leading eigenvector is \emph{sparse} then a simple thresholding estimator could yield consistent estimation.  Since then, there has been much work on generalizing the model in \citet{johnstone_consistency_2009} to settings where either the leading eigenvalues tend to infinity \citep{bao_statistical_2020,cai_limiting_2020,cai_subspace_2021,fan_asymptotic_2020,yan_inference_2021} or the leading eigenvectors are sparse \citep{amini_high-dimensional_2009,daspremont_direct_2007,cai_sparse_2013,gao_sparse_2017,gataric_sparse_2020,gu_sparse_2014,lei_sparsistency_2015,ma_sparse_2013,yang_sparse_2015}.  

\onecolumn

In this paper we consider the setting where the leading eigenvalues of the covariance matrix are bounded away from zero and infinity, but the leading $k$ eigenvectors are $s$-sparse as $n$ and $p$ tend to infinity.  There have been substantial theoretical
\citep{banks_information-theoretic_2018,cai_sparse_2013,krauthgamer_semidefinite_2015,vu_minimax_2013,wang_statistical_2016} and methodological \citep{berthet_optimal_2013,chen_new_2020,gataric_sparse_2020,ma_sparse_2013,rohe_vintage_2020,xie_bayesian_2019} developments in sparse PCA.  In \citet{vu_fantope_2013} the authors propose a semidefinite program enforcing sparsity to estimate the leading eigenvectors of the population covariance given only the sample covariance, and in \citet{lei_sparsistency_2015} the authors provide general results for which the algorithm in \citet{vu_fantope_2013} selects the correct support.  Similarly, \citet{gu_sparse_2014} propose a nonconvex algorithm that selects the correct support with high probability.   

In  many of the existing theoretical results on sparse PCA, authors are primarily concerned with subspace estimation error in spectral or Frobenius norm (e.g. \citet{cai_sparse_2013,vu_fantope_2013,vu_minimax_2013}).  However, in many situations entrywise guarantees can lead to more refined results which can be useful for downstream inference.  In this paper, building upon a host of recent works on entrywise guarantees for eigenvectors \citep{abbe_ell_p_2020,abbe_entrywise_2020,agterberg_entrywise_2021,cai_subspace_2021,cape_signal-plus-noise_2019,cape_two--infinity_2019,charisopoulos_entrywise_2020,chen_spectral_2020,damle_uniform_2020,fan_ell_infty_2018,jin_estimating_2019,lei_unified_2019,mao_estimating_2020, xia_statistical_2020,xie_bayesian_2019,xie_entrywise_2021,yan_inference_2021}, we study entrywise guarantees for sparse PCA for a very general class of models. Our main results hold for any \emph{sparsistent} algorithm, i.e. one that selects the correct support for the eigenvectors, with high probability.   Sparsistency has also been studied in other contexts in high-dimensional statistics, such as in sparse linear models \citep{fan_variable_2001,wainwright_sharp_2009,zhao_model_2006}.  See \citet{buhlmann_statistics_2011} for a more comprehensive overview.

The literature on entrywise eigenvector analysis includes a suite of tools and techniques to bound the entries of eigenvectors in ways that classical matrix perturbation theory (e.g. \citet{horn_matrix_2012,g._w._stewart_matrix_1990,bhatia_matrix_1997}) fails to address.  The Davis-Kahan Theorem \citep{yu_useful_2014} provides a useful benchmark for eigenvector analysis, but this can lead to suboptimal entrywise bounds.  The primary reason for the lack of optimality is due to the fact that the Davis-Kahan Theorem can be somewhat coarse, as it fails to take into account the probabilistic nature of empirical eigenvectors in statistical settings.  Therefore, entrywise eigenvector bounds require careful probabilistic and matrix analysis techniques that go beyond what the Davis-Kahan Theorem and classical matrix perturbation theory can do. See \citet{chen_spectral_2020} for an accessible introduction to entrywise eigenvector estimation.  The only other work on entrywise eigenvector analysis in sparse PCA is in \citet{xie_bayesian_2019}, which is a Bayesian setting under the relatively stringent spiked model.  In this paper we develop entrywise bounds for sparse PCA  under a much more general model class.  More specifically, our results hold for models satisfying a mild eigengap requirement (see Assumption \ref{assumption:eigenvalues}) that includes the spiked model.  

The rest of this paper is organized as follows.  In Section \ref{sec:2} we provide the requisite background for sparse PCA and existing results on sparsistency.  In Section \ref{sec:mainresults} we provide our main results, and Section \ref{sec:discussion} includes the discussion.   We include a sketch of our main proof in Section \ref{sec:proofs}, but the full proofs are relegated to the supplementary material.  



\subsection{Notation}
We use capital letters to denote matrices and random vectors, which will be clear from context, and lower case letters to denote fixed vectors.   We let $X_1, \dots , X_n$ denote a collection of $n$ random variables in $\R^p$. For a generic real-valued random variable $X$, its $\psi_{\alpha}$ \emph{Orlicz norm of order $\alpha$} (or just $\psi_{\alpha}$ norm)  is defined via $\|X \|_{\psi_{\alpha}} := \inf\{t > 0: \E\exp( |X|^{\alpha}/t) \leq 1\}$.  Random variables with finite $\psi_2$ norm are called \emph{subgaussian} and random variables with finite $\psi_1$ norm are called \emph{subexponential}.  More discussion on Orlicz norms is included in Appendix \ref{sec:bernstein} in the supplementary material. 

For $d_1 \geq d_2$, we define the set of matrices $U \in \R^{d_1 \times d_2}$ with orthonormal columns as $\mathbb{O}(d_1,d_2)$ and when $d = d_1 =d_2$, we denote this set as $\mathbb{O}(d)$.  We use $\| \cdot \|$ as the spectral norm on matrices and the Euclidean norm on vectors, $\| \cdot \|_F$ as the Frobenius norm, and $\|\cdot \|_{\max}$ for the maximum entry norm.  Except for the spectral norm, we write $\| \cdot \|_{p \to q}$ as the operator norm from $\ell_p \to \ell_q$; that is $\| M\|_{p\to q} := \sup_{\|x\|_p=1} \| Mx \|_q$.  Of  particular importance is the $2\to\infty$ norm, which is the maximum row norm of a matrix. Except for the maximum entry norm, we write $\|\cdot\|_{p}$ to denote the entrywise $p$ norm of a matrix viewed as a long vector.  For a matrix $M$, $\diag(M)$ extracts its diagonal, and $\tr(M)$ is its trace.  For two symmetric matrices $A$ and $B$, we write $A \succcurlyeq B$ if $A - B$ is positive semidefinite. For a matrix $M$, $M_{j\cdot}$ and $M_{\cdot i}$ denote its $j$'th row and $i$'th column respectively.  For a collection of indices $J$, $M_{JJ}$ denotes the principal submatrix of $M$ found by taking its columns and rows corresponding to indices in $J$, and for a vector $x$, $x[J]$ denotes the components of $x$ corresponding to indices in $J$.  For a matrix $M$, the operator $\text{supp}(M)$ denotes its support, i.e. the indices corresponding to nonzero components in $M$.  We denote the \emph{reduced condition number} of $\Sigma$ (with respect to the dimension $k$) as $\kappa := \frac{\lambda_1}{\lambda_k}$.

For two functions $f(n)$ and $g(n)$, we write $f(n) \lesssim g(n)$ or $f(n) = O(g(n))$ if there exists a constant $C$ such that $f(n) \leq C g(n)$ for all $n$ sufficiently large, and we write $f(n) \ll g(n)$ or $f(n) = o(g(n))$ if $f(n)/g(n) \to 0$ as $n \to \infty$.  In the proofs, a generic constant $C$ may change from line to line.

\section{SPARSE PCA AND SPARSISTENCY} \label{sec:2}
Suppose $\{X_i\}_{i=1}^{n} \in \R^p$ are mean-zero random variables with covariance matrix $\Sigma$ and eigenvalues $\lambda_1 \geq  \cdots \geq \lambda_p \geq 0$. 
Define the empirical covariance $\hat \Sigma := \frac{1}{n} \sum_{i=1}^{n} X_i X_i\t$, which is just the usual method of moments estimator.  We assume that $\Sigma$ has a \emph{sparse} $k$-dimensional leading subspace, meaning that its leading $k$ eigenvectors are $s$-sparse,  in the sense that 
there is a set  $J \subset \{1, \dots, p\}$ with cardinality at most $s$, with each eigenvector's nonzero support restricted to indices in $J$.  In the language of \citet{vu_minimax_2013}, this setting refers to \emph{row}-sparsity (as opposed to \emph{column}-sparsity).   See \citet{vu_minimax_2013} for a comparison.  We denote the $p \times k$ matrix $U$ as the matrix of $k$ orthonormal eigenvectors of $\Sigma$.   Since $U$ is assumed row-sparse, it has at most $s$ nonzero \emph{rows}.   Concretely, this means that the nonzero support of each column of $U$ is restricted to rows with indices in $J$.  A useful interpretation of the set $J$ is that it corresponds to the subset of covariates that contribute to the directions of maximum variance.  In order for $\Sigma$  to have a well-defined (sparse) leading $k$-dimensional subspace, it must have an eigengap, meaning that $\lambda_k - \lambda_{k+1} > 0$. In Section \ref{sec:mainresults}, Assumption \ref{assumption:eigenvalues} offers a slightly more quantitative condition on this eigengap.  

The \emph{sparse PCA problem} consists of estimating the matrix $U$ from the observations $\{X_i\}_{i=1}^{n}$.  There have been a number of approaches, including, but not limited to semidefinite programming \citet{amini_high-dimensional_2009,daspremont_direct_2007}, Fantope Projection and Selection algorithm \citep{vu_fantope_2013,lei_sparsistency_2015}, nonconvex approaches \citep{gu_sparse_2014}, Bayesian approaches \citep{xie_bayesian_2019}, amongst others \citep{gataric_sparse_2020,chen_new_2020,wang_tighten_2014,ma_sparse_2013}.  In this paper we consider any algorithm that selects the correct support with high probability (see Assumption \ref{assumption:sparsistency}) in an asymptotic regime where $k \ll 
s \ll n \lesssim p$.  From a practical standpoint, it is useful to consider the regime where $k$ stays fixed but $s$ tends to infinity as $n$ and $p$ at a rate $s = o(n)$.  This regime is similar to that studied in the literature on high-dimensional sparse linear models, where one assumes that the coefficients are $s$-sparse with $s \ll n$.  While it is possible to use analogous techniques to those in sparse linear models to study sparse PCA (e.g. \citet{jankova_-biased_2021}), the unsupervised problem of sparse PCA is markedly distinct from the \emph{supervised} setting of sparse linear regression, and often requires additional considerations.


Note that if $\Pi$ is a permutation matrix, then $\Pi \Sigma \Pi\t (\Pi U) = \Pi \Sigma U = \Pi U \Lambda$, where $\Lambda$ is the $k \times k$ diagonal matrix of leading eigenvalues of $\Sigma$.  This shows that $\Pi U$ are eigenvectors of $\Pi \Sigma \Pi\t$.  Therefore, given the set of nonzero indices $J$, without loss of generality, we can assume $J = \{1, \dots , s\}$ by permuting $\Sigma$ if necessary.  We can partition $\Sigma$ via
\begin{align*}
    \Sigma := \begin{pmatrix}
    \Sigma_{JJ} & \Sigma_{JJ^c} \\ \Sigma_{JJ^c}\t & \Sigma_{J^c J^c}
    \end{pmatrix};
\end{align*}
a similar partition holds for $\hat \Sigma$ and $U$. 
Under the assumption that the leading eigenvectors of $\Sigma$ are sparse, we have from the eigenvector equation that
\begin{align*}
     \Sigma  U 
     &= \begin{pmatrix}
    \Sigma_{JJ} & \Sigma_{JJ^c} \\ \Sigma_{JJ^c}\t & \Sigma_{J^c J^c}
    \end{pmatrix} \begin{pmatrix} U_J \\ 0  \end{pmatrix} = \begin{pmatrix} \Sigma_{JJ} U_J \\ \Sigma_{JJ^c}\t U_J
    \end{pmatrix} = \begin{pmatrix} U_J \\ 0 \end{pmatrix} \Lambda
\end{align*}
which shows also that $U_J$ is orthogonal to the matrix $\Sigma_{JJ^c}\t$ and that the leading $k$ eigenvectors and eigenvalues of $\Sigma_{JJ}$ are exactly the leading $k$ eigenvectors of $\Sigma$ with the zeros removed.  

An important property of any sparse PCA algorithm is identifying the support $J$ with high probability.  Suppose $\hat U$ is any estimator for $U$ (or, equivalently, $\widehat{UU\t}$ is any estimator for $UU\t$).  In this work we consider a ``debiased'' version of sparse PCA under the assumption that $\hat U$ and $U$ contain the same set of nonzero components, which implies that the estimator $\hat U$ equivalently estimates the support $J$.  We defer the particular details of this assumption to Assumption \ref{assumption:sparsistency}. Our estimator is then defined as the following modification on any sparsistent algorithm: given any set $J$, let $\tilde U_J$ be the $s \times k$ matrix of eigenvectors of the principal submatrix $\hat \Sigma_{JJ}$, and define $\tilde U := \begin{pmatrix} \tilde U_J \\ 0 \end{pmatrix}$.  
If the algorithm is sparsistent, then the correct set $J$ will be selected with high probability. In this way, the particular choice of sparse PCA algorithm can be viewed as a variable selection procedure as opposed to an estimation procedure. The full procedure is presented in Algorithm \ref{alg:spca}. 

\begin{algorithm}[t]
\caption{}
\label{alg:spca}
\begin{algorithmic}[1]
\Require Sparsistent sparse PCA algorithm \texttt{SparsePCA}, empirical covariance matrix $\hat \Sigma$\;
\State Run \texttt{SparsePCA} algorithm on $\hat \Sigma$, obtaining support set estimate $\hat J \subset \{1, \dots, p\}$.\;
\State Define $\tilde U_{\hat{J}}$ as the leading $k$ eigenvectors of $\hat \Sigma_{\hat J\hat J}$.\;
\Return Full matrix $\tilde U$, where \begin{align*}
    \tilde U_{i\cdot} = \begin{cases} (\tilde U_{\hat J})_{i\cdot} & i \in \hat J \\ 0 & i \notin \hat J \end{cases}
\end{align*}
\end{algorithmic}
\end{algorithm}

A natural question is whether sparsistent algorithms for sparse PCA exist. The answer is positive:  in Theorem 1 of \citet{lei_sparsistency_2015}, the authors provide deterministic conditions on $\Sigma$ guaranteeing that the Fantope Projection and Selection estimator  is unique and has support set $J$ with probability at least $1 - O(p^{-2})$ when $s\sqrt{\frac{\log(p)}{n}} \to 0$.  Their conditions require an error bound on $\|\hat \Sigma - \Sigma \|_{\max}$ as well as conditions on the magnitudes of the eigengaps and entries of the projection matrices.   Similarly, \citet{gu_sparse_2014} provide general conditions on $\Sigma$ (in terms of the magnitudes of the entries) so that their (nonconvex) algorithm obtains the support set $J$ with probability at least $1 - O(n^{-2})$ when $\frac{sk\log(p)}{n} \to 0$. In general, sparsistency is a property of an algorithm, and the particular structure of $\Sigma$ must be taken into account.  Therefore, our results will hold for general matrices $\Sigma$ with only mild conditions,  and can be coupled with additional structural assumptions and algorithms to yield improved recovery guarantees.

\section{MAIN RESULTS} \label{sec:mainresults}
In order to state our main results, we need a few assumptions.  Our main results will be stated for large $n$ with $p,s$ and $k$ functions of $n$.  We have the following assumption on the dimensions.

\begin{assumption}[Sample Size and Dimension]\label{assumption:dimensions}
The sample size $n$ and dimension $p$ satisfy
\begin{align*}
    s\log(p) \ll n; \qquad k \ll s.
\end{align*}
\end{assumption}
The assumption that $s\log(p) \ll n$ is weaker than the assumption $s \lesssim \sqrt{n/\log(p)}$ as is the condition in \citet{lei_sparsistency_2015} for sparsistency.  However, this still allows $p/n \to \infty$; e.g. $p = n^{c}$ for any $c \geq 1$.  The second condition $k \ll s$ is not explicitly required, but it does rule out the degenerate case $k = O(s)$, since $k \leq s$ by definition.   In many works $k = 1$ (e.g. \citet{amini_high-dimensional_2009,elsener_sparse_2019,jankova_-biased_2021}).  

The next assumption imposes the condition that whatever variable selection procedure we use selects the correct support set $J$ with high probability.

\begin{assumption}[Sparsistency]\label{assumption:sparsistency}
The algorithm is \emph{sparsistent}, meaning that with probability $1 -\delta$ the correct set $J$ is chosen.  
\end{assumption}

Note that Theorem 1 of \citet{lei_sparsistency_2015} provides sufficient conditions for  Assumption \ref{assumption:sparsistency} to hold, as does Theorem 1 of \citet{gu_sparse_2014}.   In general, this assumption is the hardest 
to check as it depends on the particular variable selection algorithm.  In \citet{lei_sparsistency_2015}, the authors show that $\delta = O(p^{-2})$ when $s\sqrt{\frac{\log(p)}{n}}\to 0$ (in addition to some other conditions omitted here).  Similarly, \citet{gu_sparse_2014} show that $\delta = O(n^{-2})$ when $\frac{s\log(p)}{n} \to 0$ (in addition to other conditions omitted here).  Typically the other conditions include some ``signal-strength'' requirements, such as the magnitudes of the entries of $\Sigma$ being sufficiently large.  The particular details for these requirements can be found in \citet{lei_sparsistency_2015} and \citet{gu_sparse_2014} respectively.  

The following assumption imposes general tail conditions on the distribution of the observations $X_1, \dots ,X_n$.



\begin{assumption}[Randomness]\label{assumption:randomness}
The variables $X_i$ are mean zero and satisfy $X_i = \Sigma^{1/2} Y_i$ for independent random variables $Y_i$ with independent coordinates with unit variance.  Furthermore, the $\psi_2$ norm of each coordinate $Y_{ij}$ satisfies $\| Y_{ij} \|_{\psi_2} = 1$ .
\end{assumption}

This assumption says that the $X_i$'s are linear combinations of $Y_i$'s whose entries are independent.  In general, assuming that each observation is a linear combination of independent random variables is a little stringent, but still common in the random matrix theory literature (e.g. \citet{el_karoui_spectrum_2010,knowles_anisotropic_2017,bao_statistical_2020,ding_spiked_2021,yang_edge_2019,yang_linear_2020}).  While a more general result may be possible, Assumption \ref{assumption:randomness} includes the setting that the $Y_i$'s are i.i.d. Gaussians with identity covariance.  

The following assumption imposes a quantitative condition on the eigengap (note that the existence of an eigengap is required for identifiability).


\begin{assumption}[Eigenvalues]\label{assumption:eigenvalues}
The top eigenvalues of $\Sigma$ satisfy
\begin{align*}
    C \lambda_1\bigg( \sqrt{ \frac{s}{n} } + \sqrt{\frac{\log(p)}{n}} \bigg)+ \frac{\lambda_{k+1}}{8} \leq \frac{\lambda_k}{8}
\end{align*}
for some sufficiently large constant $C$. In addition, for all $p$, we have that $2\lambda_{k+1} < (1-\eps)\lambda_k$ for some $\eps > \frac{1}{64}$.
\end{assumption}

The requirement $\eps > \frac{1}{64}$ is somewhat arbitrary and can be relaxed in general to any constant strictly greater than zero.  The other part of the assumption is required to obtain enough signal on the top $k$ eigenvalues of $\Sigma$, and hence $\Sigma_{JJ}$.  Furthermore, in light of Lemma \ref{lem:spectral_bound} (our principal submatrix concentration bound), this ensures that the top $k$ eigenvalues of $\hat \Sigma_{JJ}$ ``track" those of $\Sigma_{JJ}$.  In lieu of stronger assumptions, such as in a spiked model, this is the minimum requirement to guarantee that leading eigenvectors of $\hat \Sigma_{JJ}$ are well-defined.  

The main results will be stated in terms of the $2\to\infty$ norm of the difference of two matrices.  Recall that for a matrix $M \in \R^{p\times k}$, we have that
\begin{align*}
    \| M \|_{2\to\infty} &= \max_{1\leq i \leq p} \| M_{i\cdot} \|_2;
\end{align*}
that is, $\|M\|_{2\to\infty}$ is the maximum (Euclidean) row norm of the matrix $M$.  Moreover, the $2\to\infty$ norm has some attractive geometrical properties; for example, for two matrices $A$ and $B$, we have that $\| AB \|_{2\to\infty} \leq \| A \|_{2\to\infty} \|B\|$. More discussion on these relationships can be found in \citet{cape_two--infinity_2019}.

The following assumption concerns the \emph{incoherence} of the matrix $U$, which is defined as $\|U\|_{2\to\infty}$. This assumption is only included to ease interpretation and is not explicitly required. A more general -- albeit more complicated -- result is provided in the supplementary material.  

 \begin{assumption}[Incoherence and Conditioning]\label{assumption:clarity}
Suppose $\|U\|_{2\to\infty} \lesssim \big(\frac{k}{s}\big)^{1/2}$, that $k \lesssim \sqrt{s}$, and that the eigenvalues satisfy
\begin{align*}
    \lambda_{k+1} \leq \frac{\lambda}{2} < \lambda \leq \lambda_k \leq \lambda_1 \leq \kappa \lambda
\end{align*}
for some parameters $\kappa$ and $\lambda$. 
\end{assumption}

The requirement $k \lesssim \sqrt{s}$ is only needed to simplify terms. 
The incoherence assumption states that the matrix $\Sigma_{JJ}$ is incoherent in the usual sense. In this paper we do not worry about the particular incoherence constant as long as it is $O(1)$, whereas in the matrix completion literature  \citep{candes_matrix_2010,candes_power_2010,chen_noisy_2020,chen_inference_2019} one often studies the precise dependence on the incoherence constant. If one desires a more refined understanding of incoherence, our more general result in the supplementary material shows how our upper bound depends explicitly on the incoherence of $U$.   
  
In addition, Assumption \ref{assumption:clarity} should not be confused with Assumption \ref{assumption:eigenvalues} on the eigengap.
The parameter $\kappa$ is the \emph{reduced condition number} of the leading $k$-dimensional subspace of $\Sigma$, and can be much smaller than the usual (full) condition number of $\Sigma$, especially when the leading $k$ eigenvalues are of comparable order (or ``spiked'') relative to the bottom $p - k$ eigenvalues.
Assumption \ref{assumption:eigenvalues} in fact implies an upper bound on $\kappa$ of order at most $\sqrt{n/(s\log(p))}$, but it is useful to think of the setting that $\kappa = O(1)$, which corresponds to the case where the leading $k$ eigenvalues are of comparable order.  
In the setting that the eigenvalues are uniformly bounded away from zero and infinity, this assumption is not particularly strong; moreover, if the leading $k$ eigenvalues grow sufficiently fast as a function of $n$ and $p$, then the leading $k$ eigenvectors are consistent without additional assumptions.  Consequently, the primary technical condition in Assumption \ref{assumption:clarity} is on the incoherence, i.e. $\|U\|_{2\to\infty} \lesssim \big( \frac{k}{s} \big)^{1/2}$.

Before stating the main theorem, we will require some notions from subspace perturbation theory \citep{bhatia_matrix_1997,g._w._stewart_matrix_1990}. For $V,V' \in \mathbb{O}(p,k)$, the quantity
\begin{align}
    d_F(V, V') = \inf_{W \in \mathbb{O}(k)} \| V- V' W\|_F \label{frob}
\end{align}
defines a metric on $k$-dimensional subspaces invariant to choice of basis. 
Therefore, by analogy, one might wish to study the quantity
\begin{align}
    d_{2\to\infty}(V,V') := \inf_{W \in \mathbb{O}(k)} \| V - V'W\|_{2\to\infty}. \label{twoinf}
\end{align}
Unfortunately, for fixed $V,V'$, one cannot necessarily compute the minimizer in \eqref{twoinf} in closed form.  However, for fixed $V,V'$ the minimizer of \eqref{frob} 
is attained using the singular value decomposition of $V\t V'$.  That is, let $W_1 D W_2\t$ be the singular value decomposition of $V\t V'$.  Then the minimizer of  \eqref{frob}, denoted $W_*$, satisfes $W_*: = W_1 W_2\t$. 
In addition,
\begin{align*}
    d_{2\to\infty}(V,V') &\leq \| V - V' W_*\|_{2\to\infty}.
\end{align*}
Therefore, the results will be stated in terms of the \emph{existence} of an orthogonal matrix $W_* \in \mathbb{O}(k)$ that provides an upper bound for the $2\to\infty$ distance.   In the proof, we show that $W_*$ is actually a specific Frobenius-optimal orthogonal matrix. For convenience, we also include more information on subspace distances in  the supplementary material (Appendix \ref{sec:bernstein}).

%

We are now prepared to state our main result. 

\begin{theorem}\label{cor}
Suppose Assumptions \ref{assumption:dimensions}, \ref{assumption:sparsistency}, \ref{assumption:randomness}, \ref{assumption:eigenvalues}, and \ref{assumption:clarity} are satisfied, and let $\tilde U$ be the output of Algorithm \ref{alg:spca}.  Then with probability at least $1- \delta - p^{-2}$, there exists an orthogonal matrix $W_* \in \mathbb{O}(k)$ such that 
\begin{align*}
   \max_{1\leq i \leq n} \| \tilde U_{i\cdot} - (U W_*)_{i\cdot} \| 
   &\lesssim   \kappa^2\sqrt{\frac{k\log(p)}{n}} + 
   \kappa^3 \frac{s\log(p)}{n}.
\end{align*}
Consequently, if $\kappa = O(1)$, then
\begin{align*}
   \max_{1\leq i \leq n} \| \tilde U_{i\cdot} - (U W_*)_{i\cdot} \| 
   &\lesssim   \sqrt{\frac{k\log(p)}{n}}+ \frac{s\log(p)}{n}.
\end{align*}

\end{theorem}

As a brief remark, the dependence on the reduced condition number $\kappa$ here may be suboptimal and could potentially be improved -- we believe this is primarily an artifact of our proof technique and not a fundamental requirement. Recall that in the regime that the eigenvalues are bounded away from zero and infinity, when the leading $k$ eigenvalues are of comparable order, it holds that $\kappa = O(1)$.

Note that by taking $\delta = O(p^{-2})$ and the conditions in \citet{lei_sparsistency_2015} needed for sparsistency, the above bound holds with probability at least $1 - O(p^{-2})$; similarly, under the conditions needed for sparsistency in \citet{gu_sparse_2014}, one has $\delta = O(n^{-2})$, in which case the bound holds with probability at least $1 - O(n^{-2})$.


\section{DISCUSSION}
\label{sec:discussion}

In the regime that the eigenvalues are uniformly bounded away from zero and infinity in $n$, then Theorem \ref{cor} shows that we have the error rate
\begin{align*}
   \max_{1\leq i \leq n} \| \tilde U_{i\cdot} - (U W_*)_{i\cdot} \| &\lesssim    \max\bigg( \sqrt{\frac{k\log(p)}{n}}, \frac{s\log(p)}{n}\bigg).  
\end{align*}
In contrast, under the same conditions, in Frobenius norm, it has been shown in \citet{cai_sparse_2013} that the minimax rate satisfies
\begin{align*}
 \| \tilde U- U W_* \|_{F} &\lesssim  \sqrt{\frac{s\log(p)}{n}},
\end{align*}
so Theorem \ref{cor} improves upon this.  Moreover,  our result improves greatly upon the Frobenius norm bound in \cite{vu_fantope_2013}, as well as the Frobenius minimax rates studied in \citet{cai_sparse_2013} and \citet{vu_minimax_2013}. 
To the best of our knowledge, this is the first $2\to\infty$ guarantee for sparse PCA under a generic sparsistency requirement.  A similar result was found in \citet{xie_bayesian_2019} for spiked sparse covariance matrices, but here the only assumption on the spike is Assumption \ref{assumption:eigenvalues}, which is a much weaker assumption.

Our bounds can also be compared to the spiked covariance matrix setting $\Sigma = U \Lambda U\t + \sigma^2 I$, where $U$ is no longer sparse but $\lambda_k \to \infty$ in $n$ and $p$. In this setting the eigenvectors $\hat U$ of $\hat \Sigma$ are consistent in the following sense.  Define the \emph{effective rank} $\mathfrak{r}(\Sigma) := \frac{\tr(\Sigma)}{\lambda_1}$.  Theorem 1 of \citet{cape_two--infinity_2019} (see also \citet{yan_inference_2021} and \citet{cai_subspace_2021}) shows that if $\lambda_1 \gtrsim d/k$, $\mathfrak{r}(\Sigma) = o(n)$, $\kappa = O(1)$, and $\lambda_k - \sigma^2 \gtrsim \lambda_k$, then 
\begin{align*}
    \max_{1\leq i \leq n} \| \hat U_{i\cdot} - (U W_*)_{i\cdot} \| &\lesssim \sqrt{\frac{\max\{\mathfrak{r}(\Sigma), \log(d)\}}{n}} \sqrt{\frac{k^3}{p}}.
\end{align*}
Here the primary error is no longer in \emph{detecting} the leading eigenvectors (as the assumption that $\lambda_1 \gtrsim d/k$ implies large enough separation), but rather in the inherent statistical error implicit from the difference $\hat \Sigma - \Sigma$.  Our upper bound requires that $J$ is either known or can be estimated consistently (Assumption \ref{assumption:sparsistency}), so that our error depends on the inherent statistical error from $\hat \Sigma_{JJ} - \Sigma_{JJ}$.  In contrast, we do not optimize for factors of $\lambda_1$ in our upper bound, as the setting for sparse PCA typically assumes that the eigenvalues remain bounded in $n$ and $p$.  We instead need only the (milder) eigenvalue separation in Assumption \ref{assumption:eigenvalues}.  

Suppose instead of just observing $X_1,\dots,X_n \in \R^p$, one also observes response variables $Y_i \in \R$.  Consider the linear model $Y_i = X_i\t \beta +  \eps_i$, where $\eps_i$ is a mean-zero error term with variance $\sigma^2$.  Suppose one first performs unsupervised dimensionality reduction on the data matrix via sparse PCA and then computes $\hat \beta$ using ordinary least squares with the reduced data matrix.  The $2\to\infty$ bound in Theorem \ref{cor} could provide a partial answer to the out-of-sample prediction performance using a variable selection procedure.  
To be concrete, define $\hat \beta$ as the output of ordinary least squares by regressing $Y_i$ along $X \tilde U \tilde U\t$, where $\tilde U$ is the output of the sparse PCA procedure in Algorithm \ref{alg:spca} and $X$ is the $n \times p$ matrix of predictors.  Following \citet{huang_dimensionality_2020}, we can bound the risk of a new sample point $(x_*,Y_*)$ via
\begin{align*}
    \E \| Y_* - x_*\t \tilde \beta_{\mathrm{SPCA}} \|^2 | X 
    &\leq \beta\t \big( I - \tilde U \tilde U\t \big) \Sigma \big(I - \tilde U \tilde U\t\big) \beta  + \frac{\sigma^2}{n}\tr\bigg[\bigg( \frac{1}{n} \tilde U \tilde U\t X\t X \tilde U \tilde U\t\bigg)^{\dagger} \Sigma\bigg] + \sigma^2,
\end{align*}
where the first term represents the bias, the second term represents the variance, and the third term ($\sigma^2$) is the noise intrinsic to the problem.  The bias term can be expanded further via
\begin{align*}
    \beta\t \big( I - \tilde U \tilde U\t \big) \Sigma \big(I - \tilde U \tilde U\t\big) \beta 
    &=  \beta\t \big( \tilde U \tilde U\t - UU\t\big) \Sigma \big( \tilde U \tilde U\t - UU\t\big) \beta \\
    &\qquad +  2\beta\t \big( \tilde U \tilde U\t - UU\t\big) \Sigma \big(I- UU\t\big) \beta \\
    &\qquad + \lambda_{k+1} \| \beta\|_2^2.
\end{align*}
Consider the second term.  This could be bounded by noting that
\begin{align*}
    \big| \beta\t \big( \tilde U \tilde U\t - UU\t\big) \Sigma \big( \tilde U \tilde U\t - UU\t\big) \beta \big|
    &\leq \big\| \beta\t\big( \tilde U \tilde U\t - UU\t\big) \big\|_{\infty} \big\| \Sigma \big(I- UU\t\big) \beta \big\|_1 \\
    &\leq \lambda_{k+1} \| \beta\|_{\infty} \|\beta\|_1 \| \tilde U \tilde U\t - UU\t\|_{2\to\infty}.
\end{align*}
This bound has a factor of $\|\tilde U \tilde U\t - UU\t\|_{2\to\infty}$, which, while not exactly the same as what appears in Theorem \ref{cor}, is closely related to it by appealing to notions in subspace perturbation theory (see, e.g. Lemma 1 of \citet{cai_rate-optimal_2018}).  Therefore, through similar analysis, one could obtain bounds for the other bias and variance terms with respect to the eigenvalues of $\Sigma$, the quantity $\|\tilde U\tilde U\t - UU\t\|_{2\to\infty}$ and the quantities $\|\beta\|_1$ and $\|\beta\|_{\infty}$.  Consequently, these bounds would complement those in Theorem 1 of \citet{huang_dimensionality_2020} as sparse PCA is typically needed in a regime when $\mathfrak{r}(\Sigma) \gtrsim n$, whereas \citet{huang_dimensionality_2020} study the setting that $\mathfrak{r}(\Sigma) = o(n)$.  

Finally, our upper bound depends on the debiased estimator $\tilde U_J$, which is the matrix of eigenvectors of $\hat \Sigma_{JJ}$.  A key requirement is that any algorithm obtains the correct set $J$ with probability at least $1 - \delta$.  In general, one must consider the output of an optimization procedure to determine whether a specific algorithm obtains the correct set $J$.  If one additionally wanted to \emph{test} whether a certain row of $U$ is equal to zero (i.e., whether $ i \in J$), then one would need to construct a different debiased estimator as in \citet{jankova_-biased_2021} that uses the first-order necessary optimality conditions.  This procedure therefore relies heavily on the particular algorithm used, whereas our bounds hold for generic algorithms.

\section{OVERVIEW OF THE PROOF OF THEOREM \ref{cor}}

The full proof of Theorem \ref{cor} is in the supplementary material, though we include a brief overview here.  First, our main upper bound holds without Assumption \ref{assumption:clarity}, and we provide this general upper bound in Theorem \ref{thm1} (stated in the supplementary material \ref{sec:proofs}) and show how Theorem \ref{cor} can be deduced from Assumption \ref{assumption:clarity}.  To prove Theorem \ref{thm1}, we first show the following \emph{principal submatrix concentration} bound. 

\begin{restatable}[Principal Submatrix Concentration]{lemma}{spectralbound} \label{lem:spectral_bound}
Let $J$ be an index set of $\{1, ..., p\}$ of size $s$.  Then 
\begin{align*}
    \| \hat \Sigma_{JJ} - \Sigma_{JJ} \| \lesssim  \lambda_1 \bigg( \sqrt{ \frac{s}{n} } + \sqrt{\frac{\log(p)}{n}} \bigg)
\end{align*}
with probability at least $1 -O(p^{-4})$.  
\end{restatable} 

The proof is somewhat standard and primarily follows arguments detailed in \citet{wainwright_high-dimensional_2019} via $\eps$-nets and concentration, though we include it in Section \ref{sec:proofs_prelims} for completeness.  It is also very similar to a result in \citet{amini_high-dimensional_2009} for Gaussian random variables.  To the best of our knowledge, there is no general result of this form in the literature for subgaussian random vectors.  The following Lemma shows that the leading $k$ eigenvalues of $\hat \Sigma_{JJ}$ are well-separated from its bottom eigenvalues.

\begin{restatable}[Existence of an Eigengap]{lemma}{eigengaps}\label{lem:eigengaps}
Under the event in Lemma \ref{lem:spectral_bound} and Assumption \ref{assumption:eigenvalues}, the eigenvalues of $\hat \Sigma_{JJ}$ and $\Sigma_{JJ}$ satisfy
\begin{align*}
     \lambda_k - \tilde \lambda_{k+1} \geq \frac{\lambda_k -\lambda_{k+1}}{8}; &\qquad \tilde \lambda_k - \lambda_{k+1} \geq \frac{\lambda_k -\lambda_{k+1}}{8}; \\
          & \tilde \lambda_k \geq \frac{\lambda_k}{4}.
\end{align*}
Consequently, this bound holds with probability at least $1 - O(p^{-4}).$
\end{restatable}

We also note that $\lambda_{k+1}(\Sigma_{JJ}) \leq \lambda_{k+1}$ by the Cauchy interlacing inequalities \citep{horn_matrix_2012}, and the top $k$ eigenvalues of $\Sigma_{JJ}$ are the same as those of $\Sigma$ by the eigenvector equation.  These lemmas set the stage for our main analysis.  

As an immediate consequence of Lemmas \ref{lem:spectral_bound} and \ref{lem:eigengaps}, we can obtain  the following proposition concerning the spectral proximity of $U_J U_J\t$ to $\tilde U_J \tilde U_J\t$, ensuring that $\tilde U_J$ (and hence $\tilde U$) is well-defined. 

\begin{restatable}[Spectral Proximity]{proposition}{lemdk}\label{lem:dk}
Under the assumptions of Theorem \ref{thm1}, we have that
\begin{align*}
    \| U_JU_J\t - \tilde U_J\tilde U_J\t \| \lesssim  
    \frac{\lambda_1}{\lambda_k-\lambda_{k+1}} \bigg[\sqrt{\frac{s}{n}}  + \sqrt{\frac{\log(p)}{n}}\bigg]
\end{align*}
with probability at least $1 - O(p^{-4} )$.  
\end{restatable}

We use this bound several times in our subsequent analysis. After these preliminary bounds, which are restated for convenience in the supplementary material, we develop an expansion for the difference $\tilde U_J - U_JW_*$ in terms of the error matrix $(\Sigma - \hat \Sigma)$ and deterministic quantities depending only on $\Sigma$.  Informally, we show that we have the ``first-order'' approximation
\begin{align*}
\tilde U_J - U_J W_* = (\hat \Sigma_{JJ} - U_J U_J\t \Sigma_{JJ})\tilde U_J \tilde \Lambda\inv + R,
\end{align*}
where $R$ is a residual term and $\tilde \Lambda$ is the diagonal matrix of the $k$ leading eigenvalues of $\hat \Sigma_{JJ}$. Lemma \ref{lem:eigengaps} ensures that the eigenvalues of $\tilde \Lambda$ can be bounded with respect to the eigenvalues of $\Sigma$.  The residual term $R$ (the terms $T_1, T_2$, and $T_3$ in the supplementary material) is bounded in Lemmas \ref{lem:T1}, \ref{lem:T2}, and \ref{lem:J1} with tools from complex analysis \citep{greene_function_2006},  matrix perturbation theory \citep{bhatia_matrix_1997}, and high-dimensional probability \citep{wainwright_high-dimensional_2019,vershynin_high-dimensional_2018}.  

To bound the leading term in $2\to\infty$ norm, we show that it can be further decomposed into two terms, that we dub $J_1$ and $J_2$, by the decomposition
\begin{align*}
    (\hat \Sigma_{JJ} - U_J U_J\t \Sigma_{JJ})\tilde U_J  &= (\hat \Sigma_{JJ} - \Sigma_{JJ})\tilde U_J + U_{\perp} \Lambda_{\perp} U_{\perp}\t \tilde U_J \\
    :&= J_1 + J_2,
\end{align*}
where $U_{\perp}$  is the $s \times (s-k)$ matrix such that $[U_J, U_{\perp}]$ is an orthogonal matrix. The first term reflects the error from the randomness and the leading subspace $U_J$ and the second term reflects the influence of $U_{\perp}$ on  $\tilde U_J$.  

The term $J_2 = U_{\perp} \Lambda_{\perp} U_{\perp}\t \tilde U_J$ is bounded using a matrix series expansion for the matrix $\tilde U_J$ (Lemma \ref{lem:perp_bound}).  More explicitly, we define the perturbation $E := \hat \Sigma_{JJ} - U_{J} U_J\t \Sigma_{JJ} U_J U_J\t$, and we show that we can write
\begin{align*}
    \tilde U_J &= \sum_{m=0}^{\infty} E^m (U_J \Lambda U_J\t) \tilde U_J \Lambda^{-m+1}.
\end{align*}
We then analyze each term in $2\to\infty$ norm, take a union bound for the first $O(\log(n))$ terms and bound the remaining part of the series coarsely using the spectral norm.  Similar techniques have been used in \citet{cape_signal-plus-noise_2019,xie_bayesian_2019,tang_eigenvalues_2018} and \citet{tang_asymptotically_2017}, but our analysis requires additional considerations due to the fact that we do not have a mean-zero perturbation since $\E E = U_{\perp} U_{\perp}\t \Sigma_{JJ} U_{\perp} U_{\perp}\t$.  However, the matrix $EU_{J}$ \emph{is} mean-zero since $U_{\perp}\t U_J = 0$.  

The remaining term $J_1 = (\hat \Sigma_{JJ} - \Sigma_{JJ})\tilde U_J$ is then analyzed directly through its block-structure (Equation \eqref{blocks}). Letting $X$ be the $n \times p$ matrix whose rows are the observations, by Assumption \ref{assumption:randomness}, $X = Y \Sigma^{1/2}$, where $Y$ is an  $n\times p$ matrix of independent random variables with unit variance.  Then the empirical covariance $\hat \Sigma= \frac{1}{n} X\t X$ and hence 
\begin{align*}
    \hat \Sigma_{JJ} &= \frac{1}{n}\bigg(
(\Sigma^{1/2})_{JJ} Y_J\t Y_J (\Sigma^{1/2})_{JJ} + \Sigma_{JJ^c}^{1/2} Y_{J^c}\t Y_J (\Sigma^{1/2})_{JJ} \\
&\qquad + (\Sigma^{1/2})_{JJ} Y_J\t Y_{J^c} (\Sigma_{J J^c}^{1/2})\t + \Sigma_{JJ^c}^{1/2} Y_{J^c}\t Y_{J^c} (\Sigma_{JJ^c}^{1/2})\t\bigg) 
\end{align*}
where we have abused the notation
\begin{align*}
    \Sigma_{JJ^c}^{1/2} = (\Sigma^{1/2})_{JJ^c}.
\end{align*}
Above, the $n\times p$ matrix $Y$ is partitioned via $Y = [Y_J, Y_{J^c}]$, where $Y_{J}$ and $Y_{J^c}$ are the variables corresponding to $J$ and its complement, $J^c$, respectively. This term is bounded in Lemmas \ref{lem:K1}, \ref{lem:K2}, and \ref{lem:K3}.  Lemmas \ref{lem:K1} and \ref{lem:K2} are standard applications of matrix perturbation theory (via Proposition \ref{lem:dk}) and standard concentration inequalities such as Bernstein's inequality, but Lemma \ref{lem:K3} requires studying the spectral properties of the matrix $\Sigma_{JJ^c}$ and its relation to $U_J$ (Proposition \ref{prop:sigma_props}).  

Our proof is then completed by combining and aggregating all of these bounds.  Throughout the proof we make heavy use of several important concentration inequalities  and notions from subspace perturbation theory, 
so Appendix \ref{sec:bernstein} in the supplementary material contains additional information on these topics. 

\subsection*{Acknowledgements}
 Both authors thank the anonymous reviewers for feedback during the review process. JS's research was partially supported by NSF grant CCF 2007649. JA’s research was partially supported by a fellowship from the Johns Hopkins Mathematical Institute of Data Science (MINDS) via its NSF TRIPODS award CCF-1934979, and through the Charles and Catherine Counselman Fellowship.

\bibliography{sparse.bib}
\appendix 
\clearpage
\thispagestyle{empty}
\onecolumn \makesupplementtitle

\begin{abstract}
This supplementary material contains all the proofs of our main results.  Appendix \ref{sec:proofs} contains the proof of Theorem \ref{cor}, Appendix \ref{proofs:int} contains the proofs of additional lemmas needed en route the the proof of the main theorem, and Appendix \ref{sec:bernstein} contains additional background material on Orlicz norms, concentration inequalities, and subspace perturbation theory.  
\end{abstract}

\section{Proof of Theorem \ref{cor}}
\label{sec:proofs}
In this section we prove Theorem \ref{cor}.  First, Theorem \ref{cor} is actually a consequence of the following more general theorem that does not require Assumption \ref{assumption:clarity}.   Section \ref{sec:prelims} develops the preliminary bounds in terms of principal submatrix and eigenvalue concentration (Lemmas \ref{lem:spectral_bound} and \ref{lem:eigengaps}), and in Section \ref{sec:thm1proof} we prove Theorem \ref{thm1}.  In Section \ref{sec:corproof} we show how Theorem \ref{cor} can be deduced by combining Theorem \ref{thm1} with Assumption \ref{assumption:clarity}. En route to the proof of Theorem \ref{thm1} we introduce several technical lemmas; we prove these in Section \ref{proofs:int}.  Recall that we denote $\kappa := \frac{\lambda_1}{\lambda_k}$ as the (reduced) condition number of $\Sigma$.  

\begin{theorem}\label{thm1}
Suppose Assumptions \ref{assumption:dimensions}, \ref{assumption:sparsistency}, \ref{assumption:randomness}, and \ref{assumption:eigenvalues} are satisfied.  Then with probability at least $1-\delta - p^{-2}$,  there exists an orthogonal matrix $W_* \in \mathbb{O}(k)$ such that
\begin{align*}
    \max_{1\leq i \leq p} \| \tilde U_{i\cdot} - (U W_*)_{i\cdot} \| \lesssim \mathcal{E}_1 + \mathcal{E}_2 + \mathcal{E}_3 + \mathcal{E}_4 + \mathcal{E}_5 
\end{align*}
where
\begin{align*}
    \mathcal{E}_1 :&=\frac{ \kappa \lambda_1}{\lambda_k - \lambda_{k+1}} \frac{s \log(p)}{n} \|U\|_{2\to\infty}  + \kappa k \sqrt{\frac{\log(p)}{n}} \|U\|_{2\to\infty} \\
    \mathcal{E}_2 :&=  \frac{  \lambda_1 ^2}{(\lambda_k - \lambda_{k+1})^2} \frac{s \log(p)}{n} \|U \|_{2\to\infty}\\
    \mathcal{E}_3 :&=    \sqrt{\frac{s \log(p)}{n}} \frac{\kappa \lambda_1^{1/2}}{\lambda_k - \lambda_{k+1}}  \min\bigg( \|\Sigma\|_{\max}^{1/2}, \sqrt{\lambda_1} \| U \|_{2\to\infty} \bigg)\\
  \mathcal{E}_4 :&= 
 \frac{\lambda_{k+1}}{\lambda_k} \kappa^2 \sqrt{\frac{k\log(p)}{n}} + \frac{\lambda_{k+1}}{\lambda_k} \kappa^3 \frac{s\log(p)}{n}; \\
  \mathcal{E}_5 :&= \frac{\kappa \lambda_1}{\lambda_k - \lambda_{k+1}} \frac{ s \log(p)}{n} + \kappa  \sqrt{\frac{k\log(p)}{n}}.
\end{align*}
\end{theorem}

\subsection{Preliminary Bounds}\label{sec:prelims}
Note that by Assumption \ref{assumption:sparsistency}, we need only examine the $s \times k$ matrix of eigenvectors of $\hat \Sigma_{JJ}$ and $\Sigma_{JJ}$ respectively.  We will develop an expansion for the difference $\tilde U_J - U_JW_*$ by viewing $\hat \Sigma_{JJ}$ as a perturbation of $\Sigma_{JJ}$. For convenience we restate the initial preliminary bounds in the main paper.  Except for Proposition \ref{lem:dk}, the proofs are in Section \ref{sec:proofs_prelims}.  The first is the following principal submatrix concentration bound.





\spectralbound*


Henceforth, we assume the correct support set $J$ is known; it is the correct set $J$ with probability at least $1 -\delta$ by Assumption \ref{assumption:sparsistency}.  As discussed in the main paper, using Lemma \ref{lem:spectral_bound}, we can derive the following eigenvalue bound, which we present as a lemma below.  

\eigengaps*

%

Finally, we have the following $\sin\Theta$ distance error between $U_J$ and $\tilde U_J$.

\lemdk*

\begin{proof}[Proof of Proposition \ref{lem:dk}]
By the Davis-Kahan Theorem \citep{bhatia_matrix_1997,yu_useful_2014} and Lemma \ref{lem:eigengaps}, 
\begin{align}
     \| U_J U_J\t - \tilde U_J\tilde U_J\t \| &\leq \frac{\| \hat \Sigma_{JJ} - \Sigma_{JJ}\|}{\lambda_k - \tilde \lambda_{k+1}}.  \nonumber \\
     &\lesssim \frac{\| \hat \Sigma_{JJ} - \Sigma_{JJ}\|}{\lambda_k -  \lambda_{k+1}} \label{dk1}
\end{align}
By Lemma \ref{lem:spectral_bound}, with probability at least $1 - O(p^{-4})$, the numerator can be bounded by
\begin{align*}
    \| \hat \Sigma_{JJ} - \Sigma_{JJ}\| &\leq \lambda_1  \bigg( \sqrt{ \frac{s}{n} } + \sqrt{\frac{\log(p)}{n}} \bigg)
\end{align*}
Combining this and Equation \eqref{dk1} gives the result. 
\end{proof}

In the proofs that follow, we will use the fact that by Proposition  \ref{lem:dk}, we have that 
\begin{align*}
    \| U_JU_J\t - \tilde U_J\tilde U_J\t\| \lesssim \frac{\lambda_1}{\lambda_k -\lambda_{k+1}} \sqrt{\frac{s\log(p)}{n}},
\end{align*}
which is a little more amenable to downstream analysis.  In addition, we use several equivalent expressions for the spectral norm of the difference of projections; see Lemma  \ref{lem:sintheta} in Appendix \ref{sec:bernstein} 
for a discussion of how to equate these.


\subsection{Proof of Theorem \ref{thm1}} \label{sec:thm1proof}
We now proceed with the proof.  At a high level, the argument consists of a deterministic matrix decomposition, each term of which we bound in probability.  Define $\tilde \Lambda$ as the diagonal $k\times k$ matrix of eigenvalues of $\hat \Sigma_{JJ}$.  Define $W_*$ to be the matrix
\begin{align*}
    W_* := \argmin_{W \in \mathbb{O}(k)} \| \tilde U_J - U_J W \|_F.
\end{align*}
Is is well-known that $W_*$ can be computed from the singular value decomposition of $U_J\t \tilde U_J$ (e.g.  \citet{abbe_entrywise_2020,cape_two--infinity_2019,chen_spectral_2020}). 

We now expand the difference via
\begin{align}
    \tilde U_J - U_J W_*&= \tilde U_J - U_J U_J\t \tilde U_J - U_J( W_* - U_J\t \tilde U_J) \nonumber \\
    &= \tilde U_J - U_J \Lambda U_J\t \tilde U_J\tilde \Lambda\inv + U_J \Lambda U_J\t \tilde U_J\tilde \Lambda\inv -  U_J U_J\t \tilde U_J - U_J( W_* - U_J\t \tilde U_J)  \nonumber\\
    &= \tilde U_J - U_J \Lambda U_J\t \tilde U_J\tilde \Lambda\inv + U_J (\Lambda U_J\t \tilde U_J-  U_J\t \tilde U_J\tilde \Lambda )\tilde \Lambda\inv  - U_J( W_* -  U_J\t \tilde U_J) \nonumber \\
    &= \tilde U_J \tilde \Lambda \tilde \Lambda\inv - U_J \Lambda U_J\t \tilde U_J\tilde \Lambda\inv + U_J (\Lambda U_J\t \tilde U_J-  U_J\t \tilde U_J\tilde \Lambda )\tilde \Lambda\inv  - U_J( W_* -  U_J\t \tilde U_J)  \nonumber\\
    &= (\tilde U_J \tilde \Lambda-  U_J \Lambda U_J\t \tilde U_J) \tilde \Lambda\inv + U_J (\Lambda U_J\t \tilde U_J-  U_J\t \tilde U_J\tilde \Lambda )\tilde \Lambda\inv  - U_J( W_* - U_J\t \tilde U_J) \nonumber\\
    &= A + T_1 - T_2, \label{init}
\end{align}
where 
\begin{align*}
    A :&=  (\tilde U_J \tilde \Lambda-  U_J \Lambda U_J\t \tilde U_J) \tilde \Lambda\inv; \\
    T_1 :&= U_J (\Lambda U_J\t \tilde U_J-  U_J\t \tilde U_J\tilde \Lambda )\tilde \Lambda\inv  \\
    T_2 :&= U_J( W_* - U_J\t \tilde U_J).
\end{align*}
Both $T_1$ and $T_2$ are analyzed concisely in Lemmas \ref{lem:T1} and \ref{lem:T2} as follows.  Their proofs are in Section \ref{sec:proofst1t2}. The proof of  Lemmas \ref{lem:T1} and \ref{lem:T2} are both rather straightforward and based on previous results in entrywise eigenvector analysis \citep{abbe_ell_p_2020,abbe_entrywise_2020,agterberg_entrywise_2021,cai_subspace_2021,cape_signal-plus-noise_2019,cape_two--infinity_2019,chen_spectral_2020,tang_asymptotically_2017,xia_statistical_2020,xie_bayesian_2019,xie_entrywise_2021,yan_inference_2021}.  


\begin{restatable}[Bound on $T_1$]{lemma}{tone}\label{lem:T1}
We have that 
\begin{align*}
  \|   U_J (\Lambda U_J\t \tilde U_J-  U_J\t \tilde U_J\tilde \Lambda )\tilde \Lambda\inv \|_{2\to\infty} &\lesssim \frac{\|U\|_{2\to\infty} \lambda_1^2}{\lambda_k(\lambda_k - \lambda_{k+1})} \frac{s \log(p)}{n}  \\
      &\qquad  + \frac{k\lambda_1 \|U\|_{2\to\infty}}{\lambda_k}  \sqrt{\frac{\log(p)}{n}} \\
      &\equiv \mathcal{E}_1
\end{align*}
with probability at least $1 - O(p^{-4} )$.
\end{restatable}

\begin{restatable}[Bound on $T_2$]{lemma}{ttwo}\label{lem:T2}
We have that
\begin{align*}
  \|U_J( W_* - U_J\t \tilde U_J)\|_{2\to\infty} &\lesssim  \frac{ \|U \|_{2\to\infty} \lambda_1 ^2}{(\lambda_k - \lambda_{k+1})^2} \frac{s \log(p)}{n} \\
  &\equiv \mathcal{E}_2
\end{align*}
with probability at least $1 - O(p^{-4})$. 
\end{restatable}

\ \\
\noindent
\textbf{Expanding Equation \eqref{init} into $T_3$ and $T_4$:}\\ \ \\
\noindent
We further expand the remaining term in \eqref{init} by viewing $\hat \Sigma_{JJ}$ as a perturbation of $U_J U_J\t \Sigma_{JJ}$ and using the eigenvector-eigenvalue equation via
\begin{align*}
A &=     (\tilde U_J \tilde \Lambda-  U_J \Lambda U_J\t \tilde U_J) \tilde \Lambda\inv \\
&= (\hat \Sigma_{JJ} \tilde U_J -  \Sigma_{JJ} U_J U_J\t \tilde U_J) \tilde \Lambda\inv \\
    &= (  U_J U_J\t \Sigma_{JJ}  \tilde U_J+ (\hat \Sigma_{JJ} - U_J U_J\t \Sigma_{JJ} ) \tilde U_J - \Sigma_{JJ} U_J U_J\t \tilde U_J) \tilde \Lambda\inv \\
    &= U_J U_J\t \Sigma_{JJ} (\tilde U_J - U_J U_J\t \tilde U_J) \tilde \Lambda\inv + (\hat \Sigma_{JJ} -U_J U_J\t \Sigma_{JJ} )\tilde U_J \tilde \Lambda\inv \\
    &=  T_3 +  T_4,
\end{align*}
where
\begin{align*}
    T_3 &:= U_J U_J\t\Sigma_{JJ} (\tilde U_J - U_J U_J\t \tilde U_J) \tilde \Lambda\inv \\
    T_4 &:=  (\hat \Sigma_{JJ} - U_J U_J\t \Sigma_{JJ} )\tilde U_J \tilde \Lambda\inv.
\end{align*}
The term $ T_3$  can be analyzed via techniques from complex analysis.  We present this bound as a lemma, but defer the proof to Section \ref{sec:j1proof}. 
\begin{restatable}[Bound on $T_3$]{lemma}{three}\label{lem:J1}
We have that
\begin{align*}
    \| U_J U_J\t \Sigma_{JJ} (\tilde U_J - U_J U_J\t \tilde U_J) \tilde \Lambda\inv \|_{2\to\infty} &\lesssim     \sqrt{\frac{s \log(p)}{n}} \frac{\lambda_1^{3/2}}{\lambda_k (\lambda_k - \lambda_{k+1})}  \min\bigg( \|\Sigma\|_{\max}^{1/2}, \sqrt{\lambda_1} \| U \|_{2\to\infty} \bigg)\\
  &\equiv \mathcal{E}_3
\end{align*}
with probability at least $1 - O(p^{-3})$.  
\end{restatable}

\ \\
\noindent
\textbf{Expanding $T_4$ in terms of $J_1$ and $J_2$:}\\ \ \\
\noindent
We now proceed to bound $T_4$.  We have by Lemma \ref{lem:eigengaps} and properties of the $2\to\infty$ norm that
\begin{align}
   \| (\hat \Sigma_{JJ} - U_J U_J\t \Sigma_{JJ})\tilde U_J \tilde \Lambda\inv \|_{2\to\infty} &\leq \frac{1}{\tilde \lambda_k} \|(\hat \Sigma_{JJ} - U_J U_J\t \Sigma_{JJ})\tilde U_J \|_{2\to\infty} \nonumber \\
   &\lesssim \frac{1}{\lambda_k} \|(\hat \Sigma_{JJ} - U_J U_J\t \Sigma_{JJ})\tilde U_J \|_{2\to\infty}.\label{m2}
\end{align}
Note that $\tilde U_J$ is the matrix of eigenvectors of $\hat \Sigma_{JJ}$ and hence is not independent of the error matrix $\hat \Sigma_{JJ} - U_J U_J\t \Sigma_{JJ}$, so one cannot bound the maximum row norm of the matrix above with standard concentration techniques. 
Let $U_{\perp}$ be the matrix such that $[U_J, U_{\perp}]$ is an $s \times s$ orthogonal matrix, and let $\tilde U_{\perp}$ be defined similarly. Define also $\Lambda_{\perp}$ and $\tilde\Lambda_{\perp}$ as the matrix of bottom $s - k$ eigenvalues of $\Sigma_{JJ}$ and $\hat \Sigma_{JJ}$ respectively.  Since $\tilde U_{\perp}\t \tilde U_J = 0$, we have that
\begin{align}
   \frac{1}{\lambda_k} \|(\hat \Sigma_{JJ} -& U_J U_J\t \Sigma_{JJ}) \tilde U_J \|_{2\to\infty} \nonumber \\
   &=  \frac{1}{\lambda_k} \bigg\|\bigg( \tilde U_J \tilde \Lambda \tilde U_{J}\t + \tilde U_{\perp} \tilde \Lambda_{\perp} \tilde U_{\perp}\t  - U_J \Lambda_J U_J\t\bigg) \tilde U_J\bigg\|_{2\to\infty}  \nonumber
    \\
     &\leq  \frac{1}{\lambda_k} \bigg\|\bigg( \tilde U_J \tilde \Lambda \tilde U_{J}\t+\tilde U_{\perp} \tilde \Lambda_{\perp} \tilde U_{\perp}\t   - U_J \Lambda_J U_J\t  - U_{\perp} \Lambda_{\perp} U_{\perp}\t \bigg) \tilde U_J\bigg\|_{2\to\infty} + \frac{1}{\lambda_k} \| U_{\perp} \Lambda_{\perp} U_{\perp}\t  \tilde U_J\|_{2\to\infty} \nonumber
    \\
   &\leq \frac{1}{\lambda_k}\|(\hat \Sigma_{JJ} -  \Sigma_{JJ}) \tilde U_J \|_{2\to\infty} + \frac{1}{\lambda_k}\| U_{\perp} \Lambda_{\perp}  U_{\perp}\t \tilde U_J \|_{2\to\infty}  \nonumber \\
    :&=\frac{1}{\lambda_k} \| J_{1} \|_{2\to\infty} + \frac{1}{\lambda_k}\| J_{2} \|_{2\to\infty}\label{twoterms}
\end{align}
where
\begin{align*}
    J_1 :&= (\hat \Sigma_{JJ} -  \Sigma_{JJ}) \tilde U_J;  \\
    J_2 :&=   U_{\perp} \Lambda_{\perp}  U_{\perp}\t \tilde U_J.
\end{align*}
The term $J_2$  can be bounded in the following lemma, but it is rather technical; moreover, it requires some analysis that is relatively novel in the subspace estimation literature; in particular, we combine some ideas from \citet{xia_statistical_2020} as well as \citet{cape_signal-plus-noise_2019,xie_bayesian_2019,tang_eigenvalues_2018,tang_asymptotically_2017}.  The proof is in Section \ref{sec:perpboundproof}. 

\begin{restatable}[Bound on $J_2$]{lemma}{uperp}
\label{lem:perp_bound}
The term $J_2$ satisfies
\begin{align*}
    \| U_{\perp} \Lambda_{\perp}  U_{\perp}\t \tilde U_J \|_{2\to\infty} &\lesssim 
    \kappa^2 \lambda_{k+1} \sqrt{\frac{k \log (p)}{n}} + 
    \lambda_{k+1}\kappa^3 \frac{s\log(p)}{n} \\
    &\lesssim  \mathcal{E}_4 \lambda_k
\end{align*}
with probability at least $1 - O(p^{-3})$.  
\end{restatable}

\ \\
\noindent \textbf{Further expanding the term $J_1$:}\\ \ \\
\noindent
What remains is to bound the first term of \eqref{twoterms}; i.e. the term $J_1$. First, note that by Assumption \ref{assumption:randomness}, each vector $X_i \in \R^p$ is of the form $X_i = \Sigma^{1/2} Y_i$, where $\E Y_i Y_i\t = I$.  Let $X$ be the $n\times p$ matrix whose rows are the $X_i$'s; it follows that $X = Y\Sigma^{1/2}$.  Let $Y$ be partitioned via $Y = [Y_J, Y_{J^c}]$, where $Y_J$ is the $n \times s$ matrix of variables corresponding to those in $J$, and $Y_{J^c}$ is the $n \times p - s$ matrix of the other variables.  Define through slight abuse of notation the matrix $\Sigma^{1/2}_{JJ^c} := (\Sigma^{1/2})_{JJ^c}$.  With these notations in place, we observe that since $\hat \Sigma = \frac{1}{n}(X\t X)$ we have the block structure
\begin{align*}
    \hat \Sigma_{JJ} = \frac{1}{n}\bigg(
(\Sigma^{1/2})_{JJ} Y_J\t Y_J (\Sigma^{1/2})_{JJ} + \Sigma_{JJ^c}^{1/2} Y_{J^c}\t Y_J (\Sigma^{1/2})_{JJ} + (\Sigma^{1/2})_{JJ} Y_J\t Y_{J^c} (\Sigma_{J J^c}^{1/2})\t + \Sigma_{JJ^c}^{1/2} Y_{J^c}\t Y_{J^c} (\Sigma_{JJ^c}^{1/2})\t\bigg). 
\end{align*}

Therefore, we observe that
\begin{align}
   ( \hat \Sigma_{JJ} - \Sigma_{JJ} ) \tilde U_J 
    &= \frac{1}{n}\bigg(
(\Sigma^{1/2})_{JJ} (Y_J\t Y_J - n I) (\Sigma^{1/2})_{JJ} + \Sigma_{JJ^c}^{1/2} Y_{J^c}\t Y_J (\Sigma^{1/2})_{JJ} \nonumber \\
&\qquad \qquad+ (\Sigma^{1/2})_{JJ} Y_J\t Y_{J^c} (\Sigma_{J J^c}^{1/2})\t + \Sigma_{JJ^c}^{1/2} (Y_{J^c}\t Y_{J^c} - n I) (\Sigma_{JJ^c}^{1/2})\t\bigg) \tilde U_J. \label{blocks}
\end{align}
%

Here the identity matrices are of size $s$ and $p - s$ respectively in order of appearance.  In light of the structure in \eqref{blocks}, define the matrices
 \begin{align*}
    K_1 :&= \frac{1}{n}(\Sigma^{1/2})_{JJ} (Y_J\t Y_J - n I)(\Sigma^{1/2})_{JJ} \tilde U_J;\\
    K_2 :&=\frac{1}{n} \Sigma_{JJ^c}^{1/2} Y_{J^c}\t Y_J (\Sigma^{1/2})_{JJ} \tilde U_J;\\
    K_3 :&= \frac{1}{n}(\Sigma^{1/2})_{JJ} Y_J\t Y_{J^c} (\Sigma_{J J^c}^{1/2})\t \tilde U_J;\\
    K_4 :&= \frac{1}{n}\Sigma_{JJ^c}^{1/2} (Y_{J^c}\t Y_{J^c}- nI) (\Sigma_{JJ^c}^{1/2})\t\tilde U_J.  
\end{align*}
Then
\begin{align*}
    J_1 = (\hat \Sigma_{JJ} - \Sigma_{JJ}) \tilde U_J = K_1 + K_2 + K_3 + K_4.
\end{align*}

We have lemmas that bound the $2\to\infty$ norms of each of these matrices. Each of these bounds follows essentially the same set of steps:
\begin{enumerate}
    \item Bound the $2\to\infty$ norm using properties of the $2\to\infty$ norm in terms of the maximum entry.
    \item Write each entry as a sum of mean-zero subexponential random variables and use either Bernstein's inequality or the Hanson-Wright inequality (see Appendix \ref{sec:bernstein}) to bound the result.
\end{enumerate}
The proofs for these lemmas are in Sections \ref{sec:final_bounds} and \ref{sec:proofsk3k4}.  Recall that we define $\mathcal{E}_5$ via
\begin{align*}
    \mathcal{E}_5 :&= \frac{\kappa \lambda_1}{\lambda_k - \lambda_{k+1}} \frac{s\log(p)}{n} + \kappa \sqrt{\frac{k\log(p)}{n}} \\
    &\equiv \frac{1}{\lambda_k}\bigg( \frac{\lambda_1^2}{\lambda_k - \lambda_{k+1}} \frac{s\log(p)}{n}+ \lambda_1 \sqrt{\frac{k\log(p)}{n}}\bigg).  
\end{align*}

\begin{restatable}[The matrix $K_1$]{lemma}{kone}
\label{lem:K1}
The matrix $K_1$ satisfies
\begin{align*}
    \| \frac{1}{n}(\Sigma^{1/2})_{JJ} (Y_J\t Y_J - n I)(\Sigma^{1/2})_{JJ} \tilde U \|_{2\to\infty} &\lesssim \frac{\lambda_1^2}{\lambda_k - \lambda_{k+1}} \frac{ s \log(p)}{n} + \lambda_1   \sqrt{\frac{k\log(p)}{n}} \\
    &\lesssim \mathcal{E}_5 \lambda_k
\end{align*}
with probability at least $1 - O(p^{-4})$.\end{restatable}

\begin{restatable}[The matrix $K_2$]{lemma}{ktwo}
\label{lem:K2}
The matrix $K_2$ satisfies
\begin{align*}
 \| \frac{1}{n} \Sigma_{JJ^c}^{1/2} Y_{J^c}\t Y_J (\Sigma^{1/2})_{JJ} \tilde U \|_{2\to\infty} &\lesssim \lambda_1  \sqrt{\frac{k\log(p)}{n}} +\frac{\lambda_1^{2} }{\lambda_k -\lambda_{k+1}} \frac{s  \log(p)}{n} \\
 &\lesssim \mathcal{E}_5 \lambda_k
\end{align*}
with probability at least $1 - O(p^{-4})$.
\end{restatable}

\begin{restatable}[The matrices $K_3$ and $K_4$]{lemma}{kthree}
\label{lem:K3}
The matrices $K_3$ and $K_4$ satisfy
\begin{align*}
 \|\frac{1}{n}(\Sigma^{1/2})_{JJ} Y_J\t Y_{J^c} (\Sigma_{J J^c}^{1/2})\t \tilde U_J\|_{2\to\infty} &\lesssim   \frac{s \log(p)}{n} \frac{ \lambda_1^2 }{\lambda_k - \lambda_{k+1}} \\
 &\lesssim \mathcal{E}_5\lambda_k ;\\
 \| \frac{1}{n}\Sigma_{JJ^c}^{1/2} (Y_{J^c}\t Y_{J^c}- nI) (\Sigma_{JJ^c}^{1/2})\t\tilde U \|_{2\to\infty} &\lesssim  \frac{s \log(p)}{n} \frac{ \lambda_1^2 }{\lambda_k - \lambda_{k+1}} \\
 &\lesssim  \mathcal{E}_5\lambda_k
\end{align*}
with probability at least $1 - O(p^{-3})$.  
\end{restatable}
\ \\ \ \\ \noindent
\textbf{Putting it together: } \\ \ \\
\noindent
We are now ready to complete  the proof.  We have that
\begin{align*}
    \|\tilde U_J - U_JW_* \|_{2\to\infty} &\leq \bigg\| \left(\tilde{U}_{J} \tilde{\Lambda}-U_{J} \Lambda U^{\top} \hat{U}\right) \tilde{\Lambda}^{-1} \bigg\|_{2\to\infty}+\|T_{1}\|_{2\to\infty}+\|T_{2}\|_{2\to\infty} \\
   &\leq \bigg\| \left(\tilde{U}_{J} \tilde{\Lambda}-U_{J} \Lambda U^{\top} \hat{U}\right) \tilde{\Lambda}^{-1}\bigg\|_{2\to\infty} + \mathcal{E}_1 + \mathcal{E}_2 \\
   &\leq \| T_3 \|_{2\to\infty} + \| T_4 \|_{2\to\infty} + \mathcal{E}_1 + \mathcal{E}_2 \\
   &\leq \| T_4 \|_{2\to\infty} + \mathcal{E}_1 + \mathcal{E}_2 + \mathcal{E}_3\\
   &\lesssim \frac{\|J_1\|_{2\to\infty} + \|J_2\|_{2\to\infty}}{\lambda_k} + \mathcal{E}_1 + \mathcal{E}_2 + \mathcal{E}_3 \\
   &\lesssim\frac{\|J_1\|_{2\to\infty}}{\lambda_k}  + \mathcal{E}_1 + \mathcal{E}_2 + \mathcal{E}_3 + \mathcal{E}_4 \\
   &\lesssim \frac{1}{\lambda_k}\bigg(\|K_1\|_{2\to\infty} + \|K_2 \|_{2\to\infty}+\| K_3 \|_{2\to\infty}  +\| K_4 \|_{2\to\infty}  \bigg)  + \mathcal{E}_1 + \mathcal{E}_2 + \mathcal{E}_3 + \mathcal{E}_4 \\
   &\leq \mathcal{E}_1 + \mathcal{E}_2 + \mathcal{E}_3 + \mathcal{E}_4 +\mathcal{E}_5
\end{align*}
with probability at least $1 - O(p^{-3})$.  Consequently, by the union bound and Assumption \ref{assumption:sparsistency}, this bound holds with probability at least $1 - \delta - p^{-2}$ as desired.

\subsection{Proof of Theorem \ref{cor}} \label{sec:corproof}
In this section we show how Theorem \ref{cor} can be deduced from Theorem \ref{thm1}.  We simply bound $\mathcal{E}_1$ through $\mathcal{E}_5$ using the additional assumptions introduced in Assumption \ref{assumption:clarity}.

Note that under Assumption \ref{assumption:clarity}, we have that $\lambda_{k+1} \leq \frac{\lambda}{2}$ and $\lambda_k \geq \lambda$, implying that $\lambda_k - \lambda_{k+1} \geq \frac{\lambda}{2}$. In addition $\lambda_1 \leq \kappa \lambda$.  Therefore,
\begin{align*}
    \frac{\lambda_1}{\lambda_k} &\leq \frac{\kappa\lambda}{\lambda} \leq \kappa; \\
    \frac{\lambda_1}{\lambda_k - \lambda_{k+1}} &\lesssim  \frac{\kappa\lambda}{\lambda} \leq \kappa.
\end{align*}
Therefore,
\begin{align}
    \mathcal{E}_1 &= \frac{ \kappa \lambda_1}{\lambda_k - \lambda_{k+1}} \frac{s \log(p)}{n} \|U\|_{2\to\infty}  + \kappa k \sqrt{\frac{\log(p)}{n}} \|U\|_{2\to\infty} \nonumber \\
    &\lesssim \kappa^2 \frac{s \log(p)}{n} \|U\|_{2\to\infty} +\kappa k \sqrt{\frac{\log(p)}{n}} \|U\|_{2\to\infty}  \nonumber\\
    &\lesssim \kappa^2 \frac{\sqrt{sk} \log(p)}{n} + \kappa \frac{k^{3/2}}{\sqrt{s}} \sqrt{\frac{\log(p)}{n}}  \nonumber\\
    &\lesssim \kappa^2 \frac{s\log(p)}{n} + \kappa \sqrt{\frac{k\log(p)}{n}}, \label{e1}
    \end{align}
where the penultimate inequality comes from the fact that $\|U\|_{2\to\infty}\lesssim (k/s)^{1/2}$ and that $k\lesssim \sqrt{s}$. Similarly,
\begin{align}
     \mathcal{E}_2 :&=  \frac{  \lambda_1 ^2}{(\lambda_k - \lambda_{k+1})^2} \frac{s \log(p)}{n} \|U \|_{2\to\infty}  \nonumber\\
     &\lesssim \kappa^2  \frac{s \log(p)}{n} \|U \|_{2\to\infty}   \nonumber\\
     &\lesssim \kappa^2 \frac{\sqrt{sk} \log(p)}{n}  \nonumber\\
     &\lesssim \kappa^2 \frac{s\log(p)}{n}. \label{e2}
\end{align}
For $\mathcal{E}_3$,
  \begin{align}
      \mathcal{E}_3 &=    \sqrt{\frac{s \log(p)}{n}} \frac{\kappa \lambda_1^{1/2}}{\lambda_k - \lambda_{k+1}}  \min\bigg( \|\Sigma\|_{\max}^{1/2}, \sqrt{\lambda_1} \| U \|_{2\to\infty} \bigg) \nonumber\\
      &\lesssim  \sqrt{\frac{s \log(p)}{n}} \frac{\kappa \lambda_1^{1/2}}{\lambda_k - \lambda_{k+1}} \sqrt{\lambda_1} \| U \|_{2\to\infty} \nonumber \\
      &\lesssim \kappa ^2 \sqrt{\frac{s \log(p)}{n}}\| U \|_{2\to\infty}   \nonumber\\
      &\lesssim \kappa^2 \sqrt{\frac{k\log(p)}{n}} \label{e3}
  \end{align}
  since $\|U\|_{2\to\infty}\lesssim(k/s)^{1/2}$.  For $\mathcal{E}_4$, we have that since $\lambda_{k+1} < \lambda_k$, then
  \begin{align}
       \mathcal{E}_4 &= 
  \kappa^2 \frac{\lambda_{k+1}}{\lambda_k} \sqrt{\frac{k \log (p)}{n}} +  \frac{\lambda_{k+1}}{\lambda_k}\kappa^3 \frac{s\log(p)}{n}   \nonumber\\
  &\lesssim \kappa^2 \sqrt{\frac{k \log (p)}{n}}  + \kappa^3 \frac{s\log(p)}{n}. \label{e4}
  \end{align}
  Finally, for $\mathcal{E}_5$, we see that
  \begin{align}
   \mathcal{E}_5 :&= \frac{\kappa \lambda_1}{\lambda_k - \lambda_{k+1}} \frac{ s \log(p)}{n} + \kappa  \sqrt{\frac{k\log(p)}{n}}. \nonumber \\
   &\lesssim \kappa^2 \frac{ s \log(p)}{n} + \kappa \sqrt{\frac{k\log(p)}{n}}. \label{e5}
\end{align}
The condition number is always larger than 1.  Hence, combining \eqref{e1},\eqref{e2},\eqref{e3},\eqref{e4} and \eqref{e5} completes the proof.

\section{Proofs of Intermediate Lemmas} \label{proofs:int}
In this section we collect the proofs of the Lemmas needed en route to the proof of Theorem \ref{thm1}. All the lemmas are self-contained and repeated for convenience.  

\subsection{Proofs of Lemmas \ref{lem:spectral_bound} and \ref{lem:eigengaps}}\label{sec:proofs_prelims}
First, we recall the statement of Lemma \ref{lem:spectral_bound}.

\spectralbound*

\begin{proof}[Proof of Lemma \ref{lem:spectral_bound}]
The result is similar to \citet{amini_high-dimensional_2009}, but for general subgaussian ensembles as opposed to Gaussian ensembles.   The proof is standard via an $\eps$-net; we follow similarly to the proof of Theorem 6.5 in \citet{wainwright_high-dimensional_2019}.  

Let $\Delta = \hat \Sigma_{JJ} - \Sigma_{JJ}$.  First take an $1/8$-net of the $S^{s - 1}$ sphere; denote these vectors $v_1, ..., v_N$, where $N \leq 17^s$ (see Example 5.8 in \citet{wainwright_high-dimensional_2019}).  Then for any  $s$-unit vector $v$, there exists some vector $v_j$ of distance at most $\eps = \frac{1}{8}$ to $v$.  Therefore 
\begin{align*}
    \langle v, \Delta v \rangle = \langle v_j, \Delta v_j \rangle + 2 \langle (v-v_j),\Delta v_j\rangle + \langle v - v_j , \Delta(v- v_j)\rangle.
\end{align*}
Hence, we see that by the triangle inequality and Cauchy-Schwarz,
\begin{align*}
    |\langle v , \Delta v\rangle | &\leq | \langle v_j ,\Delta v_j \rangle | + 2 \|\Delta\| \|v - v_j\| \| v_j\| + \|\Delta\| \|v - v_j \|^2\\
     &\leq | \langle v_j ,\Delta v_j \rangle | + \frac{1}{2}\| \Delta \|,
\end{align*}
where the final inequality comes from the fact that $v_j$ is at most distance $\frac{1}{8}$ to $v$.  Letting $v$ denote the unit vector achieving $\sup \langle v, Qv \rangle$ and rearranging  we have that
\begin{align*}
    \| \Delta \| \leq 2 | \langle v_j, \Delta v_j \rangle| \leq  2 \max_{1\leq i \leq n} | \langle v_i , \Delta v_i \rangle |.
\end{align*}
So we therefore have that
\begin{align*}
    \E( \exp(\lambda \| \Delta \|) ) \leq \E\bigg( \exp(2\lambda \max_{1\leq i\leq N}  | \langle v_i , \Delta v_i \rangle |) \bigg) \leq \sum_{i=1}^N\bigg( \E( \exp(2\lambda \langle v_i, \Delta v_i\rangle)) + \E (\exp(-2\lambda\langle v_i, \Delta v_i \rangle )) \bigg).
\end{align*}
We now bound the mgf above, which is the primary technical difference between this and Theorem 6.5 in \citet{wainwright_high-dimensional_2019}.  Denote $X_i[J]$ as the vector $X_i$ with only the components in $J$, and let $u$ be an arbitrary unit vector.  From the assumption the $X_i$'s are iid we have that
\begin{align*}
    \E \exp (t u\t \Delta u) &= \prod_{i=1}^{n} \E_{X_i}\bigg[ \exp\bigg( \frac{t}{n} [ (X_i[J]\t u)^2 - u\t \Sigma_{JJ} u ]\bigg) \bigg] \\
    &= \bigg( \E_{X_1}\bigg[ \exp\bigg( \frac{t}{n} [ (X_1[J]\t u)^2 - u\t \Sigma_{JJ} u ]\bigg) \bigg]\bigg)^n.
\end{align*}
Let $\eps$ be a Rademacher random variable independent of $X_1$.  Then by the contraction property of Rademacher random variables, 
\begin{align*}
    \E_{X_1}\bigg[ \exp\bigg( \frac{t}{n} [ (X_1[J]\t u)^2 - u\t \Sigma_{JJ} u ]\bigg) \bigg] &\leq \E_{X_1,\eps}\bigg[ \exp\bigg( \frac{2t}{n} \eps ((X_1[J])\t u)^2\bigg) \bigg] \\
    &= \sum_{k=0}^{\infty} \frac{1}{k!} \bigg(\frac{2t}{n}\bigg)^k \E ( \eps^k (X_1[J]\t u)^{2k} ) \\
    &= 1 + \sum_{k=1}^{\infty} \frac{1}{(2k)!} \bigg(\frac{2t}{n}\bigg)^{2k} \E ( (X_1[J]\t u)^{4k} )
\end{align*}
where the first is by the series expansion for the exponential, and the second is by noting that the $\eps$ are rademacher and hence have vanishing odd moments.  

 Note that by assumption the $X_i$'s can be written as $X_i = \Sigma^{1/2} Y_i$ for some independent $Y_i$'s satisfying $\| Y_{ij}\|_{\psi_2} \leq 1$.  Then $\| \Sigma^{1/2} Y_i \|_{\psi_2} \leq \sqrt{\lambda_1} $.  Hence, by equivalence of the subgaussian norm, the moments satisfy
\begin{align*}
     \E ( (X_1[J]\t u)^{4k} ) &\leq \frac{(4k)!}{2^{2k} (2k)!} (\sqrt{8} e \lambda_1^{1/2})^{4k}.
\end{align*}
From this, we deduce
\begin{align*}
    1 + \sum_{k=1}^{\infty} \frac{1}{(2k)!} \big(\frac{2t}{n}\big)^{2k} \E ( (X_1[J]\t u)^{4k} ) &\leq 1 + \sum_{k=1}^{\infty} \frac{1}{(2k)!} \bigg(\frac{2t}{n}\bigg)^{2k} \frac{(4k)!}{2^{2k}(2k)!} (\sqrt{8} e \lambda_1^{1/2})^{4k} \\
    &\leq 1 + \sum_{k=1}^{\infty} \bigg( \frac{16 t}{n} e^2 \lambda_1 \bigg)^
    {2k}
\end{align*}
which is a geometric series.  Hence, since $\frac{1}{1-a} \leq e^{2a}$ for all $a \in [0,1/2]$, we have that
\begin{align*}
     1 + \sum_{k=1}^{\infty} \bigg( \frac{16 t}{n} e^2 \lambda_1 \bigg)^
    {2l} &\leq \exp\bigg( 2 \big[\frac{16 t}{n} e^2 \lambda_1\big]^2\bigg)
\end{align*}
for all $|t| < \frac{n}{32e^2 \lambda_1}$.  Therefore, we have shown
\begin{align*}
      \E \exp (t u\t \Delta u) &\leq \exp\bigg( 512 \frac{t^2}{n} e^4 \lambda_1^2\bigg).
\end{align*}
From here, using the sum, we have that for all $|t| < \frac{n}{64e^2 \lambda_1}$ that
\begin{align*}
      \E( \exp(t\| \Delta \|) ) &\leq  \sum_{i=1}^N\bigg( \E( \exp(2t \langle v_i, \Delta v_i\rangle)) + \E (\exp(-2t\langle v_i, \Delta v_i \rangle )) \bigg) \\
      &\leq 2N e^{2048 \frac{t^2}{n}e^4 \lambda_1^2} \\
      &\leq \exp( C \frac{ t^2 \lambda_1^2}{n} + 4s),
\end{align*}
since $2(17^s) \leq e^{4s}$.  Therefore, by the Chernoff bound,
\begin{align*}
    \p\bigg( \|\Delta \| > \eta \bigg) &\leq \exp\bigg( C \frac{ t^2 \lambda_1^2}{n} + 4s - \eta t\bigg).
\end{align*}
Minimizing over $t$ shows that
\begin{align*}
    t &= \frac{n \eta}{2 C \lambda_1^2}
\end{align*}
is the minimizer provided that $\eta < \frac{C \lambda_1}{32e^2}.$  Plugging this value of $t$ in yields
\begin{align*}
    \p\bigg( \| \Delta \| > \eta \bigg) &\leq \exp\bigg( 4s - \frac{\eta^2 n}{4 C \lambda_1^2}\bigg) \\
    &= \exp\bigg[ n \bigg( \frac{4s}{n} - \frac{\eta^2}{4 C \lambda_1^2} \bigg) \bigg].
\end{align*}
Suppose $\eta = C_2 \lambda_1 \bigg( \sqrt{\frac{s}{n}} + \sqrt{\frac{4\log(p)}{n}} \bigg)$ for some sufficiently large constant $C_2$.  Note that Assumption \ref{assumption:dimensions} ensures that this choice of $\eta$ satisfies  $\eta < \frac{C \lambda_1}{32e^2}$ since $s/n = o(1)$ and $\log(p)/n = o(1)$.  Therefore, with this choice of $\eta$, we have that
\begin{align*}
    \exp\bigg[ n \bigg( \frac{4s}{n} - \frac{\eta^2}{4 C \lambda_1^2} \bigg) \bigg] &\leq \exp( -4\log(p)) \\
    &\leq p^{-4}.  
\end{align*}
Consequently, recalling that $\Delta = \hat \Sigma_{JJ} - \Sigma_{JJ}$ we have that
\begin{align*}
    \p\bigg[ \| \hat \Sigma_{JJ} - \Sigma_{JJ} \| >  C_2 \lambda_1 \bigg( \sqrt{\frac{s}{n}} + \sqrt{\frac{4\log(p)}{n}} \bigg) \bigg] \leq p^{-4}
\end{align*}
as desired. 
\end{proof}

Again, we recall the statement of Lemma \ref{lem:eigengaps}.

\eigengaps*

\begin{proof}[Proof of Lemma \ref{lem:eigengaps}]

By Weyl's inequality, the event in Lemma \ref{lem:spectral_bound} implies that for all $1\leq i\leq s$ that
\begin{align*}
    | \lambda_i - \tilde \lambda_i | \leq C\lambda_1 \bigg( \sqrt{ \frac{s}{n} } + \sqrt{\frac{\log(p)}{n}} \bigg).
\end{align*}
Note that $\Sigma_{JJ}$ is a principal submatrix of $\Sigma$; hence its eigenvalues satisfy $\lambda_i(\Sigma_{JJ}) \leq \lambda_i$ for all  $i \geq k+1$ (when $1\leq i \leq k$ we have equality).  Therefore,
By Assumption \ref{assumption:eigenvalues}, we have that 
\begin{align*}
    \lambda_k - \tilde \lambda_{k+1} &\geq \lambda_k - \lambda_{k+1}(\Sigma_{JJ}) - C\lambda_1 \bigg( \sqrt{ \frac{s}{n} } + \sqrt{\frac{\log(p)}{n}} \bigg)\\
    &\geq \lambda_k - \lambda_{k+1} - C\lambda_1 \bigg( \sqrt{ \frac{s}{n} } + \sqrt{\frac{\log(p)}{n}} \bigg)\\
    &\geq \frac{7}{8} (\lambda_k-\lambda_{k+1}) \\
    &\geq \frac{\lambda_k -\lambda_{k+1}}{8},
\end{align*}
and similarly for $\tilde \lambda_k - \lambda_{k+1}$.  For the final bound,
\begin{align*}
    \tilde \lambda_k &\geq \lambda_k - C \lambda_1 \bigg( \sqrt{ \frac{s}{n} } + \sqrt{\frac{\log(p)}{n}} \bigg) \\
    &\geq \lambda_k - \lambda_k/8\\
    &\geq \frac{\lambda_k}{4},
\end{align*}
which completes the proof.
\end{proof}

\subsection{Proof of Lemmas \ref{lem:T1} and \ref{lem:T2}}
\label{sec:proofst1t2}
First we will recall the statement of Lemma \ref{lem:T1}.
\tone*

\begin{proof}[Proof of Lemma \ref{lem:T1}]
Note that by properties of the $2\to\infty$ norm, we have
\begin{align}
    \|   U_J (\Lambda U_J\t \tilde U -  U_J\t \tilde U_J\tilde \Lambda )\tilde \Lambda\inv \|_{2\to\infty} &\leq \| U_J \|_{2\to\infty} \| \Lambda U_J\t \tilde U -  U_J\t \tilde U_J\tilde \Lambda \| \hat\lambda_k\inv. \label{T11}
\end{align}
We note that $\lambda_k \lesssim \tilde \lambda_k$ with probability $1 - O(p^{-4})$ by Lemma \ref{lem:eigengaps}. Furthermore, by the eigenvector equation,
\begin{align*}
    \Lambda U_J\t \tilde U_J -  U_J\t \tilde U_J\tilde \Lambda &=  (U_J \Lambda)\t \tilde U_J - U_J\t \tilde U_J \tilde \Lambda \\
    &= (\Sigma_{JJ} U
    _J)\t \tilde U - U_J\t \hat \Sigma_{JJ} \tilde U_J \\
    &= U_J\t (\Sigma_{JJ} - \hat \Sigma_{JJ} )\tilde U_J.
\end{align*}
In addition,
\begin{align*}
    U_J\t (\Sigma_{JJ} - \hat \Sigma_{JJ} )\tilde U_J &= U_J\t (\Sigma_{JJ} - \hat \Sigma_{JJ} ) U_J U_J\t \tilde U_J + U_J\t (\Sigma_{JJ} - \hat \Sigma_{JJ} ) (I -U_J U_J\t) \tilde U_J.
\end{align*}
The second term satisfies
\begin{align*}
   \|  U_J\t (\Sigma_{JJ} - \hat \Sigma_{JJ} ) (I -U_J U_J\t) \tilde U_J \| \leq  \|  U_J\t (\Sigma_{JJ} - \hat \Sigma_{JJ} ) \| \| (I -U_J U_J\t) \tilde U_J \|.
\end{align*}
Note that 
\begin{align*}
    \| (\Sigma_{JJ} - \hat \Sigma_{JJ} )U_J\| &\leq \| \Sigma_{JJ} - \hat \Sigma_{JJ}  \| \|U_J \| \\
    &\leq \| \Sigma_{JJ} - \hat \Sigma_{JJ}  \|
\end{align*}
since $U_J$ has orthonormal columns.  Therefore, by Lemma \ref{lem:spectral_bound}, 
\begin{align}
     \|  U_J\t (\Sigma_{JJ} - \hat \Sigma_{JJ} ) \| 
     &\lesssim \lambda_1  \bigg( \sqrt{\frac{s\log(p)}{n}}\bigg). \label{T13}
\end{align}
Note that $\| (I - U_J U_J\t) \tilde U_J \| \lesssim \|\sin\Theta(U_J, \tilde U_J) \|\lesssim \| U_J U_J\t - \tilde U_J \tilde U_J\t \|$ (see Lemma \ref{lem:sintheta} in Appendix \ref{sec:bernstein}).
Therefore, by Proposition \ref{lem:dk}, we have that
\begin{align}
 \|(I -U_J U_J\t) \tilde U_J \| \lesssim   \frac{  \lambda_1 }{\lambda_k-\lambda_{k+1}} \bigg( \sqrt{\frac{s\log(p)}{n}}\bigg). \label{T12}
\end{align}
In summary, we have shown so far that by \eqref{T11},  \eqref{T13}, and \eqref{T12},
\begin{align*}
      \|   U_J (\Lambda U_J\t \tilde U_J -  U_J\t \tilde U_J\tilde \Lambda )\tilde \Lambda\inv \|_{2\to\infty} &\lesssim \frac{\|U\|_{2\to\infty} \lambda_1^2}{\lambda_k(\lambda_k - \lambda_{k+1})} \frac{s \log(p)}{n}  \\
      &\qquad + \frac{\|U\|_{2\to\infty}}{\lambda_k}  \| U_J\t (\Sigma_{JJ} - \hat \Sigma_{JJ}) U_J U_J\t \tilde U_J\|.
\end{align*}
Therefore, we focus on bounding the term
\begin{align*}
     \| U_J\t (\Sigma_{JJ} - \hat \Sigma_{JJ}) U_J U_J\t \tilde U_J\|.
\end{align*}
Naively, $\|U_J\t \tilde U_J\| \leq 1$ so by submultiplicativity, we have that
\begin{align*}
     \| U_J\t (\Sigma_{JJ} - \hat \Sigma_{JJ}) U_J U_J\t \tilde U_J\| &\leq  \| U_J\t (\Sigma_{JJ} - \hat \Sigma_{JJ}) U_J\|.
\end{align*}
For any indices $i$ and $k$, the entry of the above matrix can be written as
\begin{align*}
    \frac{1}{n} \sum_{l=1}^{n} \langle (U_J)_{\cdot i}, (X_l X_l\t - \E(X_l X_l\t))(U_J)_{\cdot k} \rangle &=  \frac{1}{n} \sum_{l=1}^{n} \bigg( ((U_J)_{\cdot i}\t X_l)(X_l\t (U_J)_{\cdot k}) - (U_J)_{\cdot i}\t \Sigma (U_J)_{\cdot k} \bigg).
\end{align*}
This is a sum of independent, mean-zero subexponential random variables.  Therefore, to apply the generalized Bernstein inequality (see Theorem \ref{thm:genbernstein} in Appendix \ref{sec:bernstein}), we need to find the $\psi_1$ norm of the above random variable.  By properties of the $\psi_1$ norm, we have that
\begin{align*}
    \| ((U_J)_{\cdot i}\t X_l)(X_l\t (U_J)_{\cdot k}) -(U_J)_{\cdot j}\t \Sigma (U_J)_{\cdot i} \|_{\psi_1} &\leq C \| ((U_J)_{\cdot i}\t X_l)(X_l\t (U_J)_{\cdot k}) \|_{\psi_1} \\
    &\leq C \| (U_J)_{\cdot i}\t X_l \|_{\psi_2} \| X_l\t (U_J)_{\cdot k} \|_{\psi_2} \\
    &= C \| (U_J)_{\cdot i}\t \Sigma^{1/2} Y_l \|_{\psi_2} \| Y_l\t \Sigma^{1/2} (U_J)_{\cdot k} \|_{\psi_2} \\
    &= C \sqrt{\lambda_i \lambda_k } \| (U_J)_{\cdot i}\t Y_l \|_{\psi_2} \|(U_J)_{\cdot k}\t Y_l \|_{\psi_2}\\
    &\leq C \sqrt{\lambda_j \lambda_k}\\
    &\leq C \lambda_1
\end{align*}
since $X_l = \Sigma^{1/2} Y_l$, the $(U_J)_{\cdot i}$ are the eigenvectors of $\Sigma$ and the vectors $Y$ are assumed to be subgaussian with unit $\psi_2$ norm.  Therefore, 
by the generalized Bernstein inequality (Theorem \ref{thm:genbernstein}), we have that for fixed $i,k$, that
\begin{align*}
    \p \bigg(  |\frac{1}{n} \sum_{l=1}^{n} \langle (U_J)_{\cdot i}, (X_l X_l\t - \E(X_l X_l\t))(U_J)_{\cdot k} |\rangle \geq t \bigg) &\leq 2\exp\bigg[-c_0 n \min\bigg( \frac{t^2}{( \lambda_1 )^2}, \frac{t}{ \lambda_1 } \bigg) \bigg].
\end{align*}
Since $\log(k) \ll \log(p)$, taking $t = C \lambda_1\sqrt{\frac{2\log(k) + 4\log(p)}{n}}$ for some constant $C$ yields that 
\begin{align*}
    |(U_J\t (\Sigma_{JJ} - \hat \Sigma_{JJ}) U_J )_{ik}| &\leq C \lambda_1\sqrt{\frac{2\log(k) + 4\log(p)}{n}} \\
    &\lesssim \lambda_1 \sqrt{ \frac{\log(p)}{n}}
\end{align*}
with probability at least $1 - O(p^{-4}k^{-2})$. Therefore, 
\begin{align*}
    \| U_J\t (\Sigma_{JJ} - \hat \Sigma_{JJ}) U_J\| &\leq  \| U_J\t (\Sigma_{JJ} - \hat \Sigma_{JJ}) U_J \|_F \\
    &\leq k \| U_J\t (\Sigma_{JJ} - \hat \Sigma_{JJ}) U_J\|_{\max} \\
    &\leq C k\lambda_1\sqrt{\frac{2\log(k) + 4\log(p)}{n}} \\
    &\lesssim k\lambda_1 \sqrt{\frac{\log(p)}{n}}
\end{align*}
with probability at least $1 - O(p^{-4})$ by taking a union bound over all $k^2$ entries.  Therefore, putting it all together, we see that
\begin{align*}
     \|   U_J (\Lambda U_J\t \tilde U -  U_J\t \tilde U_J\tilde \Lambda )\tilde \Lambda\inv \|_{2\to\infty} &\lesssim \frac{\|U\|_{2\to\infty} \lambda_1^2}{\lambda_k(\lambda_k - \lambda_{k+1})} \frac{s \log(p)}{n}  \\
      &\qquad  + \frac{k\lambda_1 \|U\|_{2\to\infty}}{\lambda_k}  \sqrt{\frac{\log(p)}{n}}
\end{align*}
with probability at least $1 - O(p^{-4})$ as desired.
\end{proof}

Now we prove Lemma \ref{lem:T2}. 
\ttwo*

\begin{proof}[Proof of Lemma \ref{lem:T2}]
This proof follows similarly to ideas in \citet{cape_signal-plus-noise_2019,abbe_entrywise_2020,lei_unified_2019}.

By properties of the $2\to\infty$ norm, we have
\begin{align*}
     \|U_J( W_* - U_J\t \tilde U_J)\|_{2\to\infty} &\leq \| U_J \|_{2\to\infty}\| W_* - U_J\t \tilde U_J\|.
\end{align*}
We will now analyze the term inside the spectral norm. Note that $W_*$ is the Frobenius-optimal Procrustes transformation.  Let $V_1 \Sigma V_2\t$ be the SVD of $U_J\t \tilde U_J$.  Then $\Sigma$ contains the sines of the canonical angles between $U_J$ and $\tilde U_J$ (see \citet{bhatia_matrix_1997} or \citet{g._w._stewart_matrix_1990} for details; Lemma \ref{lem:sintheta} in Appendix \ref{sec:bernstein} also contains equivalent expressions for the $\sin\Theta$ distances).  Then, letting $\theta_j$ be the canonical angles and $\sigma_j= \cos(\theta_j)$,
\begin{align*}
    \| W_* - U_J\t \tilde U_J \| &= \| V_1 V_2\t - V_1 \Sigma V_2\t \| \\
    &= \| I - \Sigma \| \\
    &= \max_{1\leq j\leq k} |(1 - \sigma_j)| \\
    &\leq \max_{1\leq j \leq k} (1 - \sigma_j^2)\\
    &= \max_j \sin^2(\theta_j) \\
    &= \| U_J U_J\t - \tilde U_J\tilde U_J\t\|^2 \\
    &\lesssim  \frac{ \lambda_1 ^2}{(\lambda_k - \lambda_{k+1})^2} \frac{s \log(p)}{n}.
\end{align*}
with probability at least $1 - O(p^{-4})$ by Proposition \ref{lem:dk}.
\end{proof}

\subsection{Proof of Lemma \ref{lem:J1}}
Recall the statement of Lemma \ref{lem:J1}.

\three*

\label{sec:j1proof}
\begin{proof}[Proof of Lemma \ref{lem:J1}]
 Note that since $U_J\t \Sigma_{JJ} = \Lambda U_J\t$, we have that
\begin{align*}
     \| U_J U_J\t \Sigma_{JJ} (\tilde U_J - U_J U_J\t \tilde U_J) \tilde \Lambda\inv \|_{2\to\infty} &\leq  \frac{\| U \|_{2\to\infty}}{\tilde \lambda_k} \|U_J\t \Sigma_{JJ} (\tilde U_J - U_J U_J\t \tilde U_J) \| \\
     &\leq \frac{\| U \|_{2\to\infty}}{\tilde \lambda_k} \|\Lambda U_J\t (\tilde U_J \tilde U_J\t- U_J U_J\t ) \|.
\end{align*}
On the other hand,
\begin{align*}
     \| U_J U_J\t \Sigma_{JJ} (\tilde U_J - U_J U_J\t \tilde U_J) \tilde \Lambda\inv \|_{2\to\infty} &\leq \frac{\| U \Lambda^{1/2}\|_{2\to\infty}}{\tilde \lambda_k} \| \Lambda^{1/2} U_J\t (\tilde U_J - U_J U_J\t \tilde U_J) \| \\
     &\leq \frac{\sqrt{ \| \Sigma\|_{\max}}}{\tilde \lambda_k} \| \Lambda^{1/2} U_J\t (\tilde U_J  \tilde U_J\t - U_J U_J\t ) \| ,
\end{align*}
where the term $\| \Sigma\|_{\max}$ comes from the fact that
\begin{align*}
  | U_J U_J\t \Sigma_{i,j} | &= | \langle (U \Lambda^{1/2})_i , (U \Lambda^{1/2})_j \rangle |,
\end{align*}
and hence that
\begin{align*}
    \| U |\Lambda|^{1/2} \|_{2\to\infty} &= \max_i \sqrt{\langle (U \Lambda^{1/2})_i , (U \Lambda^{1/2})_i \rangle}  \\
    &\leq \max_i \sqrt{ | (U_J U_J\t \Sigma)_{ii}|} \\
    &\leq \max_{i,j} \sqrt{ | \Sigma_{ij} |},
\end{align*}
since the eigenvalues of $\Sigma$ are all positive.  Therefore,
\begin{align}
      \| U_J U_J\t \Sigma_{JJ} (\tilde U_J - U_J U_J\t \tilde U_J) \tilde \Lambda\inv \|_{2\to\infty} &\leq \frac{1}{\tilde \lambda_k} \min\bigg( \sqrt{\lambda_1}\| U\|_{2\to\infty}\|\Lambda^{1/2} U_J\t (\tilde U_J \tilde U_J\t- U_J U_J\t ) \|, \nonumber\\
      &\qquad \qquad \qquad \qquad\| \Sigma \|_{\max}^{1/2}\|\Lambda^{1/2} U_J\t (\tilde U_J \tilde U_J\t- U_J U_J\t ) \|\bigg) \label{main1}
\end{align}
Therefore, what remains is to analyze
\begin{align*}
    \| \Lambda^{1/2} U_J\t (\tilde U_J\tilde U_J\t - U_J U_J\t) \|.
\end{align*}

To find this bound, we will represent the difference $\tilde U_J\tilde U_J\t - U_J U_J\t$ using the holomorphic functional calculus as done in \citet{lei_unified_2019} for the spiked Wigner matrix setting.  This technique has been used extensively in studying eigenvector perturbation; e.g. \citet{mao_estimating_2020,lei_unified_2019,koltchinskii_perturbation_2016,xia_normal_2021,wahl_note_2019,wahl_perturbation_2019}.  More specifically, let $\mathcal{C}$ denote a contour on the complex plane with real part ranging from $\lambda_{k} - \eta$ to $\lambda_1 + \eta$, and with imaginary part ranging from $-\gamma$ to $\gamma$.  Then, for a proper choice of $\eta$, the top $k$ eigenvalues of both $\Sigma_{JJ}$ and $\hat \Sigma_{JJ}$ lie in $\mathcal{C}$, and one can write the difference of the spectral projections via a complex integral
\begin{align*}
    \tilde U_J\tilde U_J\t - U_J U_J\t &= - \bigg[ \frac{1}{2\pi i} \oint_{\mathcal{C}} (\hat \Sigma_{JJ} - z I)\inv dz - \frac{1}{2\pi i} \oint_{\mathcal{C}} (\Sigma_{JJ}- z I)\inv dz \bigg]
\end{align*}
by the residue theorem (e.g. \citep{greene_function_2006}).  Using the identity $A\inv - B\inv = B\inv( A- B)A\inv$, and assuming the real number $\eta$ is chosen appropriately so that the contours are the same, the integrals can be combined to arrive at the expression
\begin{align*}
     \tilde U_J\tilde U_J\t - U_J U_J\t  &= -\frac{1}{2\pi i} \oint_{\mathcal{C}} (\Sigma_{JJ} - z I)\inv (\hat \Sigma_{JJ} - \Sigma_{JJ}) (\hat \Sigma_{JJ} - z I)\inv dz.
\end{align*}
Premultiplying by $\Lambda^{1/2} U_J\t$ yields (formally) that
\begin{align*}
     \| \Lambda^{1/2} U_J\t (\tilde U_J\tilde U_J\t - U_J U_J\t) \| &= \frac{1}{2\pi} \bigg\| \oint_{\mathcal{C}} \Lambda^{1/2} U_J\t  (\Sigma_{JJ} - z I)\inv (\hat \Sigma_{JJ} - \Sigma_{JJ}) (\hat \Sigma_{JJ} - z I)\inv dz \bigg\|.
\end{align*}
Note that the matrix is diagonalizable by the same eigenvectors as $\Sigma_{JJ}$, so that
\begin{align*}
    U_J\t (\Sigma_{JJ} - z I)\inv &= U_J\t ( U_J (\Lambda - z I)\inv U_J\t) + U_J\t (U_{\perp} (\Lambda_{\perp} - z I)\inv U_{\perp}\t \\
    &= (\Lambda - z I)\inv U_J\t
\end{align*}
by orthonormality, where $U_{\perp}$ are defined as the $s \times s$ completion of $U_J$ such that $[U_J, U_{\perp}]$ is an $s \times s$ orthogonal matrix.  Therefore, we have
\begin{align*}
     \| \Lambda^{1/2} U_J\t (\tilde U_J\tilde U_J\t - U_J U_J\t) \| &= \frac{1}{2\pi} \bigg\| \oint_{\mathcal{C}} \Lambda^{1/2}  ( \Lambda - z I)\inv U_J\t (\hat \Sigma_{JJ} - \Sigma_{JJ}) [\tilde U , \tilde U_{\perp}] (\hat \Lambda_{all} - z I)\inv dz \bigg\|,
\end{align*}
where $\hat \Lambda_{all}$ is the diagonal matrix of all the eigenvalues of $\hat \Sigma_{JJ}$.

The rest of the proof mirrors closely that of Lemma A.8 in \citet{lei_unified_2019}.   Recall that in order to do all these manipulations, we required that the parameter $\eta$ was chosen such that the contour $\mathcal{C}$ contains the top $k$ eigenvalues of $\hat \Sigma_{JJ}$ and $\Sigma_{JJ}$.  In fact, Lemma \ref{lem:eigengaps} shows that the choice 
\begin{align*}
    \eta := \frac{\lambda_{k} - \lambda_{k+1}}{4}
\end{align*}
suffices.  To see this, note that by Lemmas \ref{lem:spectral_bound} and \ref{lem:eigengaps}, 
\begin{align*}
    | \tilde \lambda_k - \lambda_k | &\leq \frac{\lambda_k - \lambda_{k+1}}{8}; \\
    |\tilde \lambda_{k+1} - \lambda_{k+1}| &\leq \frac{\lambda_k - \lambda_{k+1}}{8},
\end{align*}
so that the interval $\lambda_k \pm \eta$ contains $\tilde \lambda_k$, the interval $\lambda_k \pm \eta$ does not contain $\tilde \lambda_{k+1}$, and both $\tilde \lambda_k$ and $\tilde \lambda_{k+1}$ satisfy
\begin{align*}
    |\lambda_k - \tilde \lambda_{k} - \eta| \geq \eta/2 \\
    |\lambda_k - \tilde \lambda_{k+1} - \eta| \geq \eta/2.
\end{align*}

Therefore, the top $k$ eigenvalues of $\hat \Sigma_{JJ}$ lie within $\mathcal{C}$ with high probability and the bottom eigenvalues lie outside of it.  With this particular choice of $\eta$, we can proceed to bound the integrand along the contour $\mathcal{C}$.  We will conduct the rest of the analysis assuming that this event holds; it does with probability at least $1 - O(p^{-4})$.  

Define $a := \lambda_k - \eta$ and $b := \lambda_1 + \eta$.   We decompose the contour $\mathcal{C}$ into the following sets
\begin{align*}
    \mathcal{C}_1 := \{z = a + x i, x \in (-\gamma,\gamma)\} &\qquad \mathcal{C}_2 := \{ z  = x + \gamma i: x \in [a,b]\} \\
     \mathcal{C}_3 := \{z = b + x i, x \in (-\gamma,\gamma)\} &\qquad \mathcal{C}_4 := \{ z  = x - \gamma i: x \in [a,b]\}.
\end{align*}
  Let $\mathcal{I}(z)$ be the integrand.  Observe that
\begin{align*}
    \bigg\| \oint_{\mathcal{C}} \mathcal{I}(z)dz \bigg\| &\leq \oint_{\mathcal{C}_1} \bigg\| \mathcal{I}(z) dz \bigg\| + \oint_{\mathcal{C}_2} \bigg\| \mathcal{I}(z) dz \bigg\| + \oint_{\mathcal{C}_4} \bigg\| \mathcal{I}(z) dz \bigg\| + \oint_{\mathcal{C}_4} \bigg\| \mathcal{I}(z) dz \bigg\|.
\end{align*}
Therefore, we bound the above integrals directly.  The tricky analysis will be along $\mathcal{C}_1$ and $\mathcal{C}_3$; we will show that the integral along $\mathcal{C}_2$ and $\mathcal{C}_4$ tend to zero for large $\gamma$.  To this end, we will focus on $\mathcal{C}_1$ first. 
Note that
\begin{align}
 \oint_{\mathcal{C}_1} \bigg\|\Lambda^{1/2} & ( \Lambda - z I)\inv U_J\t (\hat \Sigma_{JJ} - \Sigma_{JJ}) [\tilde U , \tilde U_{\perp}] (\hat \Lambda_{all} - z I)\inv \bigg\| dz \label{c1} \\
 &\leq \oint_{\mathcal{C}_1} \bigg\|\Lambda^{1/2}  ( \Lambda - z I)\inv \bigg\| \bigg\| U_J\t (\hat \Sigma_{JJ} - \Sigma_{JJ}) [\tilde U , \tilde U_{\perp}] \bigg\| \bigg\| (\hat \Lambda_{all} - z I)\inv \bigg\| dz  \nonumber\\
 &\leq \bigg\| U_J\t (\hat \Sigma_{JJ} - \Sigma_{JJ}) [\tilde U , \tilde U_{\perp}] \bigg\|\int_{-\gamma}^{\gamma} \bigg\|\Lambda^{1/2}  ( \Lambda - (a+xi) I)\inv \bigg\|  \bigg\| (\hat \Lambda_{all} - (a+xi) I)\inv \bigg\| dx.\nonumber
\end{align}
 First, recall the definition of $a := \lambda_k - \eta$.  The term on the right-most side satisfies
\begin{align*}
     \bigg\| (\hat \Lambda_{all} - (a+xi) I)\inv \bigg\| &\leq \frac{1}{\sqrt{(\eta)^2/4 + x^2}}
\end{align*}
for all $x$ since $(\hat \lambda_i - a) \geq \eta/2$.  Therefore, we are left to bound the middle term, for which we must bound
\begin{align*}
    \max_{1\leq i \leq k} \frac{\lambda_i^{1/2}}{\sqrt{(\lambda_i - a)^2 + x^2}}.
\end{align*}
Define the function 
\begin{align*}
    g(u;x,a) := \frac{u}{\sqrt{(u - a)^2 + x^2}}.
\end{align*}
Then
\begin{align*}
     \max_{1\leq i \leq k} \frac{\lambda_i^{1/2}}{\sqrt{(\lambda_i - a)^2 + x^2}} &\leq \sup_{u \geq a + \eta} \bigg(g(u;x,a)\bigg)^{1/2}\frac{1}{(\eta^2 + x^2)^{1/4}}.  
\end{align*}
The details of the function $g$ are carried out in \citet{lei_unified_2019}; the analysis therein implies
\begin{align*}
    \sup_{u\geq a + \eta}g(u;x,a) &\leq \frac{a + \eta}{\sqrt{\eta^2 + x^2}} \mathbb{I}_{|x| \leq \sqrt{a\eta}} + \sqrt{\frac{a + \eta}{\eta}}\mathbb{I}_{|x| > \sqrt{a\eta}}.
\end{align*}
Therefore the integral from \eqref{c1} satisfies
\begin{align*}
\int_{-\gamma}^{\gamma} &\| \Lambda^{1/2} ( \Lambda - (a + x i)I)\inv \| \| (\hat \Lambda_{all} - (a + x i)I)\inv \| dx \\ &\leq  \int_{-\gamma}^{\gamma} \frac{1}{\sqrt{\eta^2/4 + x^2}} \frac{1}{(\eta^2 + x^2)^{1/4}}  \bigg( \frac{a + \eta}{\sqrt{\eta^2 + x^2}} \mathbb{I}_{|x| \leq \sqrt{a\eta}} + \sqrt{\frac{a + \eta}{\eta}}\mathbb{I}_{|x| > \sqrt{a\eta}}\bigg)^{1/2}dx \\
    &\leq   \int_{|x| \leq \sqrt{a\eta}}  \frac{4}{(\eta^2 + x^2)^{3/4}} \bigg( \frac{a + \eta}{\sqrt{\eta^2 + x^2}} \bigg)^{1/2} dx + \int_{|x| > \sqrt{a\eta}} \frac{4}{(\eta^2 + x^2)^{3/4}} \bigg( \sqrt{\frac{a + \eta}{\eta}} \bigg)^{1/2} dx \\
    &\leq  4\sqrt{a + \eta} \int_{|x| \leq \sqrt{a\eta}} \frac{1}{\eta^2 + x^2} dx + 4 \bigg(\frac{a+\eta}{\eta}\bigg)^{1/4} \int_{|x| > \sqrt{a\eta}} \frac{1}{(\eta^2 + x^2)^{3/4}} dx \\
    &\leq  8\sqrt{a + \eta} \int_{0}^{\sqrt{a\eta}} \frac{1}{\eta^2 + x^2} dx + 8 \bigg(\frac{a+\eta}{\eta}\bigg)^{1/4} \int_{\sqrt{a\eta}}^{\infty} \frac{1}{(\eta^2 + x^2)^{3/4}} dx \\
    &\leq  8\frac{\sqrt{a + \eta}}{\eta} \int_{0}^{\sqrt{a/\eta}} \frac{1}{1 + u^2} du + 8 \bigg(\frac{a+\eta}{\eta}\bigg)^{1/4} \frac{1}{\eta^{1/2}} \int_{\sqrt{a/\eta}}^{\infty} \frac{1}{(1 + u^2)^{3/4}} du \\
    &\leq  8\frac{\sqrt{a + \eta}}{\eta} 2\pi  + 8 \bigg(\frac{a+\eta}{\eta}\bigg)^{1/4} \frac{1}{\eta^{1/2}} \int_{\sqrt{a/\eta}}^{\infty} \frac{1}{ u^{3/2}} du \\
    &\leq  16\pi \frac{\sqrt{a + \eta}}{\eta}   + 8 \bigg(\frac{a+\eta}{\eta}\bigg)^{1/4} \frac{1}{\eta^{1/2}} \frac{2}{(a/\eta)^{1/2}} \\
    &\lesssim  \frac{\sqrt{a + \eta}}{\eta} +\bigg( \frac{a+\eta}{a} \bigg)^{1/4} \frac{1}{a^{1/2}}.
\end{align*}
Recall that $a + \eta = \lambda_k $; $\eta = (\lambda_{k} - \lambda_{k+1})/4$.  With these, the bound becomes (up to constants)
\begin{align*}
 \oint_{\mathcal{C}_1} \bigg\|\Lambda^{1/2} (\Lambda - zI)\inv U_J\t &(\hat \Sigma_{JJ} - \Sigma_{JJ} (\tilde U, \tilde U_{\perp})(\hat \Lambda_{all} - zI)\inv\bigg\|  dz  \\
&\lesssim  \frac{\sqrt{\lambda_1}}{\lambda_k - \lambda_{k+1}} \| (\hat \Sigma_{JJ} - \Sigma_{JJ})U_J \|+ \kappa^{1/4} \frac{1}{\lambda_k^{1/2}}\| (\hat \Sigma_{JJ} - \Sigma_{JJ})U_J \| \\
  &\lesssim \frac{ \sqrt{\lambda_1}}{\lambda_k-\lambda_{k+1}}\| (\hat \Sigma_{JJ} - \Sigma_{JJ})U_J \|.
\end{align*}
The exact same argument goes through for contour $\mathcal{C}_3$ as well. We will see that the contours along the imaginary axis tend to zero as $\gamma \to \infty$.  Assuming this for the moment, by Equation \eqref{main1}, we see that the final bound is of the form
\begin{align*}
\frac{1}{\lambda_k} \| \Lambda^{1/2} U_J\t(\tilde U_J\tilde U_J\t - U_J U_J\t) &\| \min\bigg( \| \Sigma\|_{\max}^{1/2}, \sqrt{\lambda_1}\|U\|_{2\to\infty}\bigg) \\
&\lesssim  \frac{\| (\hat \Sigma_{JJ} - \Sigma_{JJ})U_J \|}{\lambda_k} \bigg(  \frac{\sqrt{\lambda_1}}{\lambda_{k} - \lambda_{k+1}} \bigg) \min\bigg( \|\Sigma\|_{\max}^{1/2}, \sqrt{\lambda_1} \| U \|_{2\to\infty} \bigg)  
\end{align*}  
By Lemma \ref{lem:spectral_bound}, we have that the term $\| (\hat \Sigma_{JJ} - \Sigma_{JJ})U_J \|$ can be bounded via
\begin{align*}
    \lambda_1 \sqrt{\frac{s \log(p)}{n}}
\end{align*}
with probability at least $1 - O(p^{-4})$.  Therefore, the bound becomes
\begin{align*}
    \sqrt{\frac{s \log(p)}{n}} \frac{\lambda_1^{3/2}}{\lambda_k (\lambda_k - \lambda_{k+1})}  \min\bigg( \|\Sigma\|_{\max}^{1/2}, \sqrt{\lambda_1} \| U \|_{2\to\infty} \bigg),
\end{align*}
which is the desired bound.

It remains to show that the  integrals tend to zero along the curves $\mathcal{C}_2$ and $\mathcal{C}_4$. Let $\mathcal{I}(z)$ denote the integrand. Then for sufficiently large $\gamma$,
\begin{align*}
     \oint_{\mathcal{C}_2}  \| \mathcal{I}(z) \| dz  &=  \int_{a}^{b}\bigg\| \Lambda^{1/2} (\Lambda- (x+\gamma i) I)\inv U_J\t (\hat \Sigma_{JJ} - \Sigma_{JJ})[ \tilde U_J \tilde U_{\perp}](\hat \Lambda_{all} - (x+\gamma i)I)\inv \bigg\|dx  \\
     &\leq (b - a) \sup_{x \in [a,b]} \bigg\| \Lambda^{1/2} (\Lambda- (x+\gamma i) I)\inv U_J\t (\hat \Sigma_{JJ} - \Sigma_{JJ})[ \tilde U_J \tilde U_{\perp}](\hat \Lambda_{all} - (x+\gamma i)I)\inv \bigg \| \\
     &= O(\gamma^{-2}),
\end{align*}
which tends to zero as $\gamma \to\infty$.  \end{proof}

\subsection{Proof of Lemma \ref{lem:perp_bound}} \label{sec:perpboundproof}
First, recall the statement of Lemma \ref{lem:perp_bound}.
\uperp*

Recall the definition of $J_2$ via
\begin{align*}
    J_2 := U_{\perp} \Lambda_{\perp}  U_{\perp}\t \tilde U_J.
\end{align*}
Again $U_{\perp}$ is the matrix such that the $s\times s$ matrix $[U_J, U_{\perp}]$ is orthogonal. 

\begin{proof}[Proof of Lemma \ref{lem:perp_bound}]
Define the matrix $E := \hat \Sigma_{JJ} - U_J U_J\t \Sigma_{JJ} U_J U_J\t$.   Note that
\begin{align*}
    \tilde U_J \Lambda - E \tilde U_J = U_J U_J\t \Sigma_{JJ} U_J U_J\t \tilde U_J.
\end{align*}
Following \citet{cape_signal-plus-noise_2019} (see also \citet{xie_bayesian_2019,tang_asymptotically_2017,tang_eigenvalues_2018}), by Assumption \ref{assumption:eigenvalues}, the spectra of $E$ and $\Lambda$ are disjoint almost surely, so the matrix $\tilde U$ can be expanded as a matrix series (Theorem VII.2.2 in \citet{bhatia_matrix_1997}) via
\begin{align*}
    \tilde U_J &= \sum_{m=0}^{\infty} E^m (U_J\Lambda U_J\t) \tilde U_J \Lambda^{-(m+1)}.
\end{align*}
Therefore,
\begin{align*}
  J_2 =   U_{\perp} \Lambda_{\perp} U_{\perp}\t \tilde U_J &= U_{\perp} \Lambda_{\perp} U_{\perp}\t E U \Lambda U\t \tilde U \Lambda^{-2} + \sum_{m=2}^{\infty} U_{\perp} \Lambda_{\perp} U_{\perp}\t E^m U \Lambda U_J\t \tilde U_J \Lambda^{-(m+1)}
\end{align*}
since the $0$-th term cancels off because $U_{\perp}\t U_J = 0$.  Taking the first term and setting $R$ to be the rest of the series, we have that,
\begin{align}
  \|  U_{\perp} \Lambda_{\perp} U_{\perp}\t \tilde U_J \tilde U_J\t \|_{2\to\infty} &= \|U_{\perp} \Lambda_{\perp} U_{\perp}\t E U_J \Lambda U_J\t \tilde U \Lambda^{-2} \|_{2\to\infty}+ \|R\|_{2\to\infty}, \label{residual}
\end{align}
where $R$ is the residual to be bounded.  We first bound the leading term.  We have that
\begin{align}
\|U_{\perp} \Lambda_{\perp} U_{\perp}\t E U_J \Lambda U_J\t \tilde U_J \Lambda^{-2} \|_{2\to\infty} &\leq \|U_{\perp} \Lambda_{\perp} U_{\perp}\t E U_J\|_{2\to\infty} \lambda_k\inv \kappa . \label{noresid}
\end{align}
We note that since $U_{\perp}\t U_J = 0$, then
\begin{align*}
    EU_J = (\hat \Sigma_{JJ} - U_J U_J\t\Sigma_{JJ} U_J U_J\t)U_J = (\hat \Sigma_{JJ} - \Sigma_{JJ})U_J.
\end{align*}
Define $\Sigma_{JJ}^{\perp} :=U_{\perp} \Lambda_{\perp} U_{\perp}\t.$  In light of the block structure in \eqref{blocks}, we see that we can write $\sjp EU_J$ via the sum of the terms
\begin{align*}
\frac{1}{n} (\sjp) \bigg((\Sigma^{1/2})_{JJ}& (Y_J\t Y_J -nI) (\Sigma^{1/2})_{JJ} + \Sigma_{JJ^c}^{1/2} Y_{J^c}\t Y_J (\Sigma^{1/2})_{JJ} \\
&+ (\Sigma^{1/2})_{JJ} Y_J\t Y_{J^c} (\Sigma_{J J^c}^{1/2})\t 
+ \Sigma_{JJ^c}^{1/2} (Y_{J^c}\t Y_{J^c} - nI) (\Sigma_{JJ^c}^{1/2})\t\bigg)U_J.
\end{align*}
Recalling that $(\Sigma_{JJ^c}^{1/2})\t U_J = 0$, this yields the only the terms
\begin{align*}
\frac{1}{n} (\sjp) \bigg((\Sigma^{1/2})_{JJ} (Y_J\t Y_J -nI) (\Sigma^{1/2})_{JJ} + \Sigma_{JJ^c}^{1/2} Y_{J^c}\t Y_J (\Sigma^{1/2})_{JJ} \bigg)U_J &= \sjp (\Sigma^{1/2})_{JJ}\bigg( \frac{ Y_J\t Y_J}{n} - I\bigg) U_J \Lambda^{1/2} \\&\qquad + \sjp \Sigma_{JJ^c}^{1/2} \frac{Y_{J^c}\t Y_J}{n} U_J \Lambda^{1/2}.
\end{align*}
Define $A_{3/2} := \sjp (\Sigma^{1/2})_{JJ}$, which satisfies $\|A_{3/2} \| \leq \sqrt{\lambda_1} \lambda_{k+1}$.   In $2\to\infty$ norm, we have that
\begin{align*}
  \|  A_{3/2}\bigg( \frac{ Y_J\t Y_J}{n} - I\bigg) U_J \Lambda^{1/2} \|_{2\to\infty} &\leq \sqrt{k\lambda_1} \max_{i,j} \bigg| \bigg( A_{3/2}\bigg( \frac{ Y_J\t Y_J}{n} - I\bigg) U_J\bigg)_{ij} \bigg|.
\end{align*}
Define the matrix $M$ via $M_{kl} := (A_{3/2})_{ik} U_{lj}$.  Fixing $i$ and $j$, note that we can write the $i,j$ entry above as
\begin{align*}
\bigg| \sum_{k,l} M_{kl} \bigg(\frac{1}{n} ( \sum_{q=1}^{n} (Y_{ql} Y_{qk} - \E Y_{ql} Y_{qk}) \bigg) \bigg| &=\frac{1}{n} \bigg| \sum_{q}  \sum_{k,l} M_{kl}\bigg( Y_{ql} Y_{qk} - \E Y_{ql} Y_{qk} \bigg) \bigg| \\
&\leq \frac{1}{n}  \sum_{q}  \bigg| \sum_{k,l} M_{kl}\bigg( Y_{ql} Y_{qk} - \E Y_{ql} Y_{qk} \bigg) \bigg| \\
&\leq \max_{q}  \bigg| \sum_{k,l} M_{kl}\bigg( Y_{ql} Y_{qk} - \E Y_{ql} Y_{qk} \bigg) \bigg|,
\end{align*}
which is a quadratic form in the random variables $Y_{ql}$ (for fixed $q$).   To bound this, we will apply the Hanson-Wright inequality (Theorem \ref{thm:hw} in Appendix \ref{sec:bernstein}), which requires bounding the Frobenius norm of $M$.  Note that we can bound the Frobenius norm of $M$ via
\begin{align*}
\| M \|_F^2 &= \sum_{k,l} M_{kl}^2 \\
&= \sum_{k,l} (A_{3/2})_{ik}^2 U_{lj}^2 \\
&= \| A_{3/2}\|_{2\to\infty}^2 \\
&\leq \bigg(\sqrt{\lambda_1} \lambda_{k+1} \bigg)^2.
\end{align*}
Therefore, applying the Hanson-Wright inequality shows that
\begin{align*}
\p\bigg( \bigg| \sum_{k,l} M_{kl}\bigg( Y_{ql} Y_{qk} - \E Y_{ql} Y_{qk} \bigg) \bigg| > t \bigg) \leq 2 \exp\bigg( -c \min\bigg\{ \frac{t^2}{\|M\|_F^2}, \frac{t}{\|M\|} \bigg\} \bigg).
\end{align*}
Set $t := C\sqrt{\frac{\log(s)+\log(k) + 5\log(p)}{n}} \sqrt{\lambda_1} \lambda_{k+1}$.  Then since $\frac{\log(p)}{n} = o(1)$, we see that with probability at least $1 - O(s\inv k\inv p^{-5})$ that
\begin{align*}
\bigg| \sum_{k,l} M_{kl}\bigg( Y_{ql} Y_{qk} - \E Y_{ql} Y_{qk} \bigg) \bigg| &\lesssim \sqrt{\lambda_1} \lambda_{k+1} \sqrt{\frac{\log(p)}{n}}.
\end{align*}
Taking a union bound over all $n$ random variables shows that with probability at least $1 - O(s\inv k\inv p^{-4})$,
\begin{align*}
\sqrt{k\lambda_1} \bigg| \bigg( A_{3/2}\bigg( \frac{ Y_J\t Y_J}{n} - I\bigg) U_J\bigg)_{ij} \bigg| &\lesssim \lambda_1 \lambda_{k+1} \sqrt{\frac{k\log(p)}{n}}.
\end{align*}
Taking a union bound over all $s$ rows and $k$ columns yields that with probability at least $1 - O(p^{-4})$,
\begin{align}
     \|  A_{3/2}\bigg( \frac{ Y_J\t Y_J}{n} - I\bigg) U_J \Lambda^{1/2} \|_{2\to\infty}  &\lesssim \lambda_{k+1}\lambda_1  \sqrt{\frac{k \log(p)}{n}}. \label{r1}
\end{align}
For the other term, proceeding similarly,
\begin{align*}
\| \sjp \Sigma_{JJ^c}^{1/2} \frac{Y_{J^c}\t Y_J}{n} U_J \Lambda^{1/2} \|_{2\to\infty} &\leq \sqrt{\lambda_1 k} \max_{i,j} \bigg| \bigg( (\Sigma_{JJ}^{\perp} \Sigma_{JJ^c}^{1/2}) \frac{Y_{J^c}\t Y_J}{n} U_J \bigg)_{ij} \bigg| \\
&\leq \sqrt{\lambda_1 k}  \max_{i,j} \max_q \bigg|\sum_{k=s+1}^{p} \sum_{l=1}^{s} (\Sigma_{JJ}^{\perp} \Sigma_{JJ^c}^{1/2})_{ik} Y_{qk} Y_{ql} (U_J)_{lj} \bigg|.
\end{align*}
For fixed $q$, $i$, and $j$, note that $k$ ranges from $s+1$ to $p$ and $l$ ranges from $1$ to $s$, so this is a sum of independent exponential random variables. We will bound these using  Bernstein's inequality (Theorem \ref{thm:genbernstein} in Appendix \ref{sec:bernstein}).  Note that the $\ell_2$ norm of the coefficients is bounded by 
\begin{align*}
\sum_{k=s+1}^{p} \sum_{l=1}^{s} (\Sigma_{JJ}^{\perp} \Sigma_{JJ^c}^{1/2})_{ik}^2 (U_J)_{lj}^2 &\leq \| \sjp \Sigma_{JJ^c}^{1/2} \|_{2\to\infty}^2 \leq (2 \sqrt{\lambda_1} \lambda_{k+1})^2.
\end{align*}
Similarly,
\begin{align*}
\max_{k,l}| (\Sigma_{JJ}^{\perp} \Sigma_{JJ^c}^{1/2})_{ik} (U_J)_{lj} | &\leq \|U_J\|_{2\to \infty} \max_{i,k}| e_i\t (\Sigma_{JJ}^{\perp} \Sigma_{JJ^c}^{1/2}) e_k | \\
&\leq 2 \| U_J \|_{2\to\infty} \sqrt{\lambda_1} \lambda_{k+1}.
\end{align*}
By the generalized Bernstein Inequality (Theorem \ref{thm:genbernstein}), we have for any fixed $i$,$j$, and $q$ that 
\begin{align*}
    \p\bigg(  \bigg|\sum_{k=s+1}^{p} \sum_{l=1}^{s} (\Sigma_{JJ}^{\perp} \Sigma_{JJ^c}^{1/2})_{ik} Y_{qk} Y_{ql} (U_J)_{lj} \bigg| > t\bigg) &\leq 2 \exp\bigg[-c \min\bigg( \frac{t^2}{(\sqrt{\lambda_1} \lambda_{k+1})^2}, \frac{t}{\|U\|_{2\to\infty}\sqrt{\lambda_1} \lambda_{k+1}} \bigg) \bigg].
\end{align*}
Taking $t = \sqrt{\lambda_1} \lambda_{k+1} \sqrt{\frac{\log(s) +\log(k) 5\log(p)}{n}}$ shows that this holds with probability at least $1- O(s\inv k\inv p^{-5})$. Taking a union bound over $s$ rows, $k$ columns, and $n$ different random variables shows that with probability at least $1 - O(p^{-4})$ that
\begin{align}
\| \sjp \Sigma_{JJ^c}^{1/2} \frac{Y_{J^c}\t Y_J}{n} U_J \Lambda^{1/2} \|_{2\to\infty} &\leq \sqrt{\lambda_1 k} \max_{i,j} \bigg| \bigg( (\Sigma_{JJ}^{\perp} \Sigma_{JJ^c}^{1/2}) \frac{Y_{J^c}\t Y_J}{n} U_J \bigg)_{ij} \bigg| \nonumber \\
&\lesssim \lambda_{k+1} \lambda_1 \sqrt{\frac{k\log(p)}{n}} \label{r2}
\end{align}
Combining  \eqref{r2} and \eqref{r1} with \eqref{noresid} yields that
\begin{align}
   \| U_{\perp} \Lambda_{\perp} U_{\perp}\t E U_J \Lambda  U_J\t \tilde U_J \Lambda^{-2} U_J\t  \|_{2\to\infty} &\lesssim \frac{\kappa}{\lambda_k} \|U_{\perp} \Lambda_{\perp} U_{\perp}\t E U_J\|_{2\to\infty} \nonumber \\
   &\lesssim \frac{\kappa }{\lambda_k} \bigg( \lambda_1 \lambda_{k+1} \sqrt{\frac{k\log(p)}{n}} \bigg) \nonumber\\
   &\lesssim \kappa^2 \lambda_{k+1} \sqrt{\frac{k\log(p)}{n}}. \label{term1bound}
\end{align}
So what remains is to bound the residual term $R$ in \eqref{residual}.  Recall the definition of $R$ via
\begin{align*}
    R := \sum_{m=2}^{\infty} U_{\perp} \Lambda_{\perp} U_{\perp}\t E^m U_J \Lambda U_J\t \tilde U_J \Lambda^{-(m+1)}.
\end{align*}
We will bound for each $m$, but for clarity, we will first study the case $m = 2$.  We have that
\begin{align*}
    U_{\perp} \Lambda_{\perp} U_{\perp}\t E^2 U_J &= U_{\perp} \Lambda_{\perp}U_{\perp}\t ( \hat \Sigma_{JJ} - U_J U_J\t \Sigma_{JJ} U_J U_J\t )( \hat \Sigma_{JJ} - U_J U_J\t \Sigma_{JJ} U_J U_J\t )U_J \\
    &= U_{\perp} \Lambda_{\perp} U_{\perp}\t( \hat \Sigma_{JJ} - U_J U_J\t \Sigma_{JJ} U_J U_J\t )( \hat \Sigma_{JJ} - \Sigma_{JJ} )U_J \\
    &= U_{\perp} \Lambda_{\perp} U_{\perp} (\hat \Sigma_{JJ} - \Sigma_{JJ})^2 U_J + (U_{\perp} \Lambda_{\perp} U_{\perp})^2 (\hat \Sigma_{JJ} - \Sigma_{JJ})U_J.
\end{align*}
The first term is readily bounded  by Lemma \ref{lem:spectral_bound}, and the second term can be bounded using the techniques in the previous part of the proof of this Lemma.  

We now generalize this strategy for each $m$, by first providing a similar identity to the one above. 
Define $\Delta := \hat \Sigma_{JJ} - \Sigma_{JJ}$.  Note that by definition $E = \Delta + U_{\perp} \Lambda_{\perp} U_{\perp}\t$ and that $E U_J = \Delta U_J$.  Then we have that
\begin{align}
   U_{\perp} \Lambda_{\perp} U_{\perp}\t E^m U_J &= U_{\perp} \Lambda_{\perp}\t U_{\perp}E^{m-1} \Delta U_J \nonumber \\
    &= U_{\perp} \Lambda_{\perp} U_{\perp}\t E^{m-2} (\Delta + U_{\perp} \Lambda_{\perp} U_{\perp}) \Delta U_J \nonumber \\
   &= U_{\perp} \Lambda_{\perp} U_{\perp}\t E^{m-2} \Delta U_J + U_{\perp} \Lambda_{\perp} U_{\perp}\t E^{m-2} U_{\perp} \Lambda_{\perp} U_{\perp}\t \Delta U_J. \label{induction}
\end{align}
%
%
Let $\mathfrak{s}(m)$ be the set of indices such that $s_1 + ... + s_{m} = m$.  Then for all $m$ we have that 
\begin{align*}
   U_{\perp} \Lambda_{\perp} U_{\perp}\t E^m U_J &= U_{\perp} \Lambda_{\perp} U_{\perp}\t \bigg[ \sum_{\mathfrak{s}(m)} \Delta^{s_1} (U_{\perp} \Lambda_{\perp} U_{\perp}\t)^{s_2} \Delta^{s_3} (U_{\perp} \Lambda_{\perp} U_{\perp}\t)^{s_4}\cdots (U_{\perp} \Lambda_{\perp} U_{\perp}\t)^{s_{m-1}}\Delta^{s_m} \bigg] U_J,
\end{align*}
which is essentially a noncommutative Binomial Theorem.

First, consider the case that $s_1, ..., s_m$ has only single powers of $\Delta$ appearing.   If $\Delta$ appears all the way on the right hand side; that is, $s_m = 1$, then for any integer $m_0$, we have that
\begin{align*}
\| U_{\perp} \Lambda_{\perp}^{m_0} U_{\perp}\t \Delta U_J \|_{2\to\infty} &\leq C \lambda_{k+1}^{m_0} \bigg( \lambda_1 \sqrt{\frac{k\log(p))}{n}}  \bigg),
\end{align*}
with probability at least $1- O(p^{-4})$ using analogous techniques to the steps leading up to Equation \eqref{term1bound} (i.e. the case $m_0 =1$).  If $\Delta$ is not on the right hand side, suppose that its index is $s_{g} = 1$.  Then this term is of the form
\begin{align*}
   (U_{\perp} \Lambda_{\perp} U_{\perp}\t)^{1+ s_1 + s_2 + ... + s_{g-1}} \Delta  (U_{\perp} \Lambda_{\perp} U_{\perp}\t)^{s_{g+1} + ... + s_{m_0}}U_J \equiv 0
\end{align*}
since $U_{\perp}\t U_J = 0$. So the only terms that  have at most one factor of $\Delta$ appearing are those that show up as $\Delta U_J$.

Next, if $s_1, ..., s_m$ is a set of integers and at least two of the terms $s_i$ that appear on the $\Delta$ factor are greater than 1, then
\begin{align*}
    \| U_{\perp} \Lambda_{\perp} U_{\perp}\t \Delta^{s_1} (U_{\perp} \Lambda_{\perp} U_{\perp}\t)^{s_2} \Delta^{s_3} (U_{\perp} \Lambda_{\perp} U_{\perp}\t)^{s_4}\cdots (U_{\perp} \Lambda_{\perp} U_{\perp}\t)^{s_{m-1}}\Delta^{s_m} U_J \|_{2\to\infty} &\leq \| \Delta \|^{2} \lambda_{k+1}^{m-1},
\end{align*}
provided that $\|\Delta \| < \lambda_{k+1}$, which happens by Assumption \ref{assumption:dimensions} and the spectral norm concentration in Lemma \ref{lem:spectral_bound} with probability at least $1 - O(p^{-4})$.  Fix this event.  Then for any $m$, there are at most $2^m$ ways to select exponents with a power of at least two on the term  $\|\Delta\|$. Therefore, this implies that for fixed $m$
\begin{align*}
 \|  U_{\perp } \Lambda_{\perp} U_{\perp}\t E^m U_J \|_{2\to\infty} &\leq \| U_{\perp} \Lambda_{\perp}^{m} U_{\perp}\t \Delta U_J \|\\&\qquad  + \sum_{\{m: \text{exponent on $\|\Delta\|$ is at least 2}\}}   \| U_{\perp} \Lambda_{\perp} U_{\perp}\t \Delta^{s_1}\cdots (U_{\perp}\Lambda_{\perp} U_{\perp}\t)^{s_{m-1}}\Delta^{s_m} U_J \|_{2\to\infty} 
 \\
&\leq C \lambda_{k+1}^m \bigg( \lambda_1 \sqrt{\frac{k\log(p))}{n}} \bigg) + 2^m \lambda_{k+1}^{m-1} \| \Delta \|^2.
\end{align*}
This bound corresponds to one such $m$, and hence is its own event.  In order to bound for all $m$, we follow a strategy in \citet{xia_statistical_2020}.  Let $\tilde m := \lceil \log(p) \rceil$.  Then
\begin{align*}
  \|  \sum_{m=2}^{\infty} &U_{\perp} \Lambda_{\perp} U_{\perp}\t E^m U_J \Lambda U_J\t \tilde U_J \Lambda^{-(m+1)} \|_{2\to\infty} \\
  &\leq \sum_{m=2}^{\infty} \| U_{\perp} \Lambda_{\perp} U_{\perp}\t E^m U_J \|_{2\to\infty} \frac{\lambda_1}{\lambda_k^{m+1}}
 \\
 &\leq 
    \sum_{m=2}^{{\tilde m}} \| U_{\perp} \Lambda_{\perp} U_{\perp}\t E^m U_J \|_{2\to\infty} \frac{\lambda_1}{\lambda_k^{m+1}} +  \sum_{m=\tilde m}^{\infty} \| U_{\perp} \Lambda_{\perp} U_{\perp}\t E^m U_J \|_{2\to\infty} \frac{\lambda_1}{\lambda_k^{m+1}} \\
   &\leq 
  \sum_{m=2}^{\tilde m} \bigg(C \lambda_{k+1}^m \bigg( \lambda_1 \sqrt{\frac{k\log(p))}{n}} \bigg) \frac{\lambda_1}{\lambda_{k}^{m+1}} \\&\qquad +  \sum_{m=2}^{\tilde m} \bigg(2^m \lambda_{k+1}^{m-1} \| \Delta \|^2 \bigg) \frac{\lambda_1}{\lambda_{k}^{m+1}} \\
   &\qquad +  \sum_{m=\tilde m}^{\infty } \frac{\lambda_1}{\lambda_k^{m+1}} \| \Delta \| \lambda_{k+1}^{m+1}.
  \end{align*}
Define 
\begin{align*}
    \delta_1 :&= C \kappa \bigg( \lambda_1 \sqrt{\frac{k\log(p))}{n}} \bigg) \\
    \delta_2 :&= \kappa \lambda_k\inv \| \Delta \|^2
\end{align*}
Then the three  sums above can be written as
\begin{align*}
  \delta_1   \sum_{m=2}^{\tilde m} \frac{\lambda_{k+1}^m}{\lambda_k^{m}} + \delta_2\sum_{m=2}^{\tilde m}  \frac{2^m \lambda_{k+1}^{m-1}}{\lambda_k^{m-1}} + \lambda_1 \|\Delta \| \sum_{m=\tilde m}^{\infty } \frac{\lambda_{k+1}^{m+1}}{\lambda_k^{m+1}} &\lesssim   \delta_1 \frac{\lambda_{k+1}^2}{\lambda_k^2} + \delta_2 (1+\eps)\frac{\lambda_{k+1}}{\lambda_k} + \lambda_1 \| \Delta \| \bigg( \frac{\lambda_{k+1}}{\lambda_k} \bigg)^{\log(p)} \\
  &\lesssim  \delta_1 \frac{\lambda_{k+1}^2}{\lambda_k^2}  +  \delta_2 \frac{\lambda_{k+1}}{\lambda_k} + \lambda_1^2 \sqrt{\frac{s\log(p)}{n}} (1-\eps)^{\log(p)}.
\end{align*}
Here, the penultimate inequality follows from the fact that by Assumption \ref{assumption:eigenvalues}, we have that for some $\eps > 1/64$, $2\lambda_{k+1}/\lambda_k < 1 - \eps$, and hence the second term's geometric series converges.  The final inequality follows from the assumption $\lambda_{k+1}/\lambda_k < (1-\eps)$. Note that this event holds with probability at least 
$1 - O(\log(p) p^{-4}) \geq 1 -O(p^{-3})$.
Noting that
\begin{align*}
    \| \Delta \| \lesssim \lambda_1 \sqrt{\frac{s\log(p)}{n}}
\end{align*}
by Lemma \ref{lem:spectral_bound}, we see that the resulting bound for the residual satisfies
\begin{align*}
\| R \|_{2\to\infty} &\lesssim \delta_1 \big( \frac{\lambda_{k+1}}{\lambda_k} \big)^2 + \delta_2\frac{\lambda_{k+1}}{\lambda_k} + \lambda_1^2 \sqrt{\frac{s\log(p)}{n}} (1-\eps)^{\log(p)} \\
&\lesssim  \big( \frac{\lambda_{k+1}}{\lambda_k} \big)^2 \kappa  \bigg( \lambda_1 \sqrt{\frac{k\log(p))}{n}} \bigg) + \frac{\lambda_{k+1}}{\lambda_k}\kappa \lambda_k\inv \| \Delta \|^2 + \lambda_1^2 \sqrt{\frac{s\log(p)}{n}} (1-\eps)^{\log(p)} \\
&\lesssim \big( \frac{\lambda_{k+1}}{\lambda_k} \big)^2 \kappa  \bigg( \lambda_1 \sqrt{\frac{k\log(p))}{n}}  \bigg) + \frac{\lambda_{k+1}}{\lambda_k}\kappa \lambda_k\inv \| \Delta \|^2 \\
&\lesssim \big( \frac{\lambda_{k+1}}{\lambda_k} \big)^2 \kappa  \bigg( \lambda_1 \sqrt{\frac{k\log(p))}{n}}\bigg)+ \frac{\lambda_{k+1}}{\lambda_k}\kappa \lambda_k\inv \lambda_1^2 \frac{s\log(p)}{n} \\
&\lesssim \kappa^2 \lambda_{k+1} \sqrt{\frac{k\log(p)}{n}} + \lambda_{k+1} \kappa^3 \frac{s\log(p)}{n}
\end{align*}
with probability at least $1 - O(p^{-3})$ by the assumption $\eps > \frac{1}{64}$.  Combining with our initial bound in \eqref{term1bound}, we see that
\begin{align*}
    \| J_2 \|_{2\to\infty} 
    &\lesssim \kappa^2 \lambda_{k+1} \sqrt{\frac{k \log (p)}{n}} + 
    \lambda_{k+1}\kappa^3 \frac{s\log(p)}{n}
\end{align*}
with probability at least $1 - O(p^{-3})$ as desired.
\end{proof}

\subsection{Proof of Lemmas \ref{lem:K1} and \ref{lem:K2}} \label{sec:final_bounds}
Recall the statement of Lemma \ref{lem:K1}.

\kone*

Recall $K_1$ is given by
\begin{align*}
    K_1 :&= \frac{1}{n}(\Sigma^{1/2})_{JJ} (Y_J\t Y_J - n I)(\Sigma^{1/2})_{JJ} \tilde U.
\end{align*}

\begin{proof}[Proof of Lemma \ref{lem:K1}]
Note that since $U_JU_J\t + U_{\perp} U_{\perp}\t = I$, we have that
\begin{align}
   \| \frac{1}{n}(\Sigma^{1/2})_{JJ} (Y_J\t Y_J - n I)(\Sigma^{1/2})_{JJ} \tilde U_J \|_{2\to\infty}&\leq   \| \frac{1}{n}(\Sigma^{1/2})_{JJ} (Y_J\t Y_J - n I)(\Sigma^{1/2})_{JJ} U_J U_J\t \tilde U_J \|_{2\to\infty} \nonumber\\
   &\qquad + \|\frac{1}{n}(\Sigma^{1/2})_{JJ} (Y_J\t Y_J - n I)(\Sigma^{1/2})_{JJ}U_{\perp} U_{\perp}\t \tilde U_J  \|_{2\to\infty} \nonumber\\
   &\leq \| \frac{1}{n}(\Sigma^{1/2})_{JJ} (Y_J\t Y_J - n I)(\Sigma^{1/2})_{JJ} U_J \|_{2\to\infty} \| U_J\t \tilde U_J \| \nonumber\\
   &\qquad + \|\frac{1}{n}(\Sigma^{1/2})_{JJ} (Y_J\t Y_J - n I)(\Sigma^{1/2})_{JJ}U_{\perp} \|_{2\to\infty} \| U_{\perp}\t \tilde U_J \|\nonumber\\
   &\leq  \sqrt{k} \| \frac{1}{n}(\Sigma^{1/2})_{JJ} (Y_J\t Y_J - n I)(\Sigma^{1/2})_{JJ} U_J  \|_{\max}\nonumber\\
   &\qquad + \sqrt{s} \|\frac{1}{n}(\Sigma^{1/2})_{JJ} (Y_J\t Y_J - n I)(\Sigma^{1/2})_{JJ}U_{\perp} \|_{\max} \| U_{\perp}\t \tilde U_J \|, \label{bigbadboy}
\end{align}
We bound each term inside the max norm, using a strategy similar to the beginning of the proof of Lemma \ref{lem:perp_bound}.    For the first term, note that we can write the absolute value of its $i,j$ entry via
\begin{align*}
\bigg| \frac{1}{n} \sum_{q} \sum_{k,l}  &\bigg( (\Sigma^{1/2})_{JJ} \bigg)_{ik} ( Y_{qk} Y_{ql} - \E Y_{qk} Y_{ql} )  ) \bigg( (\Sigma^{1/2})_{JJ} U_J\bigg)_{kj} \bigg| \\
&\leq \max_{q} \bigg|  \sum_{k,l}  \bigg( (\Sigma^{1/2})_{JJ} \bigg)_{ik} ( Y_{qk} Y_{ql} - \E Y_{qk} Y_{ql} )  ) \bigg( (\Sigma^{1/2})_{JJ} U_J\bigg)_{lj} \bigg|.
\end{align*}
We focus on bounding for fixed $q$. This is a quadratic form in the random variable $\{Y_{qk}\}_{k=1}^{s}$.  Define the matrix $M$ via
\begin{align*}
M_{kl} := \bigg( (\Sigma^{1/2})_{JJ} \bigg)_{ik}\bigg( (\Sigma^{1/2})_{JJ} U_J\bigg)_{lj}.
\end{align*}
Note that 
\begin{align*}
\| M \|_F^2 &= \sum_{k,l} \bigg( (\Sigma^{1/2})_{JJ} \bigg)_{ik}^2 \bigg( (\Sigma^{1/2})_{JJ} U_J\bigg)_{lj}^2 \\
&\leq \lambda_1 \| (\Sigma^{1/2})_{JJ} \|_{2\to\infty}^2 \\
&\leq \lambda_1^2.
\end{align*}
Therefore, for any fixed $q$, $i$, and $j$, applying the Hanson-Wright inequality (Theorem \ref{thm:hw} in Appendix \ref{sec:bernstein}),
\begin{align*}
\p\bigg( \bigg|  \sum_{k,l}  \bigg( (\Sigma^{1/2})_{JJ} \bigg)_{ik} ( Y_{qk} Y_{ql} - \E Y_{qk} Y_{ql} )  ) \bigg( (\Sigma^{1/2})_{JJ} U_J\bigg)_{lj} \bigg| > t \bigg) &\leq 2 \exp\bigg( -c \min\bigg\{ \frac{t^2}{\lambda_1^2}, \frac{t}{\| M \|} \bigg\} \bigg).
\end{align*}
Setting $t =C \lambda_1 \sqrt{\frac{\log(s) + \log(k) + 5\log(p)}{n}}$ and taking a union bound for all $n$ random variables shows that with probability at least $1 - O(s\inv k\inv p^{-4})$ that
\begin{align*}
\max_q \bigg|  \sum_{k,l}  \bigg( (\Sigma^{1/2})_{JJ} \bigg)_{ik} ( Y_{qk} Y_{ql} - \E Y_{qk} Y_{ql} )  ) \bigg( (\Sigma^{1/2})_{JJ} U_J\bigg)_{lj} \bigg| &\lesssim \lambda_1 \sqrt{\frac{\log(p)}{n}}.
\end{align*}
Therefore, taking a union bound over all $s$ rows and $k$ columns shows that  with probability at least $1 - O(p^{-4})$ that
\begin{align}
\| \frac{1}{n}(\Sigma^{1/2})_{JJ} (Y_J\t Y_J - n I)(\Sigma^{1/2})_{JJ} U_J  \|_{\max} &\lesssim \lambda_1 \sqrt{\frac{\log(p)}{n}}. \label{bigboy1}
\end{align}
The exact same argument yields with the same probability that
\begin{align}
\|\frac{1}{n}(\Sigma^{1/2})_{JJ} (Y_J\t Y_J - n I)(\Sigma^{1/2})_{JJ}U_{\perp} \|_{\max}&\lesssim \lambda_1 \sqrt{\frac{\log(p)}{n}}. \label{bigboy2}
\end{align}
Combining \eqref{bigbadboy} with \eqref{bigboy1} and \eqref{bigboy2} yields
\begin{align*}
 \| \frac{1}{n}(\Sigma^{1/2})_{JJ} (Y_J\t Y_J - n I)(\Sigma^{1/2})_{JJ} \tilde U_J \|_{2\to\infty}&\leq \sqrt{k} \| \frac{1}{n}(\Sigma^{1/2})_{JJ} (Y_J\t Y_J - n I)(\Sigma^{1/2})_{JJ} U_J  \|_{\max}\nonumber\\
   &\qquad + \sqrt{s} \|\frac{1}{n}(\Sigma^{1/2})_{JJ} (Y_J\t Y_J - n I)(\Sigma^{1/2})_{JJ}U_{\perp} \|_{\max} \| U_{\perp}\t \tilde U_J \| \\
   &\lesssim \lambda_1 \sqrt{\frac{k\log(p)}{n}} + \lambda_1 \sqrt{\frac{s\log(p)}{n}} \| U_{\perp}\t \tilde U_J \|.
   \end{align*}
    So what remains is to bound the term $\|U_{\perp}\t \tilde U_J\|$,  However, we note that this is simply (by a factor of $\sqrt{2}$) the $\sin\Theta$ distance between the subspace $U_J U_J\t$ and $\tilde U_J\tilde U_J\t$ (see Lemma \ref{lem:sintheta} in Appendix \ref{sec:bernstein}).  Therefore, by Proposition \ref{lem:dk}, we have that this can be bounded by
\begin{align*}
\|U_{\perp}\t \tilde U_J\| \lesssim \frac{\lambda_{1}}{\lambda_{k}-\lambda_{k+1}} \sqrt{\frac{s \log ( p)}{n}}.
\end{align*}
Putting it all together, this yields that with probability at least $1 - O(p^{-4})$ that
\begin{align*}
    \| K_1 \|_{2\to\infty} &=  \| \frac{1}{n}(\Sigma^{1/2})_{JJ} (Y_J\t Y_J - n I)(\Sigma^{1/2})_{JJ} \tilde U_J \|_{2\to\infty} \\
    & \lesssim \frac{\lambda_1^2}{\lambda_k - \lambda_{k+1}} \frac{ s \log(p)}{n} + \lambda_1   \sqrt{\frac{k\log(p)}{n}},
\end{align*}
which is the desired bound.
\end{proof}

Again, we repeat the statement of Lemma \ref{lem:K2}.

\ktwo*

Recall that
\begin{align*}
     K_2 :&= \Sigma_{JJ^c}^{1/2} Y_{J^c}\t Y_J (\Sigma^{1/2})_{JJ} \tilde U_J
\end{align*}
\begin{proof}[Proof of Lemma \ref{lem:K2}]
We have that
\begin{align}
    \|\frac{1}{n} \Sigma_{JJ^c}^{1/2} Y_{J^c}\t Y_J (\Sigma^{1/2})_{JJ} \tilde U_J \|_{2\to\infty} &\leq   \|\frac{1}{n} \Sigma_{JJ^c}^{1/2} Y_{J^c}\t Y_J  (\Sigma^{1/2})_{JJ} U_J\|_{2\to\infty} \nonumber \\
    &\qquad +  \|\frac{1}{n} \Sigma_{JJ^c}^{1/2} Y_{J^c}\t Y_J  (\Sigma^{1/2})_{JJ} U_{\perp} \|_{2\to\infty}  \| U_{\perp}\t \tilde U_J\| \nonumber \\
    &\leq \sqrt{k} \| \frac{1}{n} \Sigma_{JJ^c}^{1/2} Y_{J^c}\t Y_J  (\Sigma^{1/2})_{JJ} U_J\|_{\max} \nonumber \\
    &\qquad + \sqrt{s} \|\frac{1}{n} \Sigma_{JJ^c}^{1/2} Y_{J^c}\t Y_J  (\Sigma^{1/2})_{JJ} U_{\perp} \|_{\max}  \| U_{\perp}\t \tilde U_J\|. \label{bigbaby}
\end{align}
We bound each norm inside the max separately.  Define the random variable 
 $\eta_{ij}$ as the $i,j$ entry of the matrix $\Sigma_{JJ^c}^{1/2} Y_{J^c}\t Y_J  (\Sigma^{1/2})_{JJ} U_J$.  Then 
\begin{align*}
    \eta_{ij} &= \frac{1}{n} \sum_{q=1}^{n} \sum_{k=1}^{s} \sum_{l=1}^{p-s} [ \Sigma_{JJ^c}^{1/2}]_{il} \xi^{(q)}_{s+l,k}\bigg( (\Sigma^{1/2})_{JJ} U_J \bigg)_{kj},
\end{align*}
where $\xi^{(q)}_{s+l,k} := Y_{q,s+l} Y_{qk}$.  Following a strategy similar to the proof of Lemma \ref{lem:perp_bound}, we have to bound both the maximum and sum of squared $\psi_1$ norms of the random variable
\begin{align*}
  \alpha_{qlj}:=  \frac{1}{n} [ \Sigma_{JJ^c}^{1/2}]_{il} \xi^{(q)}_{s+l,k} \bigg( (\Sigma^{1/2})_{JJ} U_J \bigg)_{kj}.
\end{align*}
The squared entries satisfy
\begin{align*}
      \| \frac{1}{n} [ \Sigma_{JJ^c}^{1/2}]_{il} \xi^{(q)}_{s+l,k} \bigg( (\Sigma^{1/2})_{JJ} U_J \bigg)_{kj} \|_{\psi_1}^2 &\leq \frac{1}{n^2} ([ \Sigma_{JJ^c}^{1/2}]_{il})^2 \bigg( (\Sigma^{1/2})_{JJ} U_J \bigg)_{kj}^2.
\end{align*}
Summing up over $q,l,j$,  
\begin{align*}
    \sum_{q=1}^{n} \sum_{k=1}^{s} \sum_{l=1}^{p-s} \| \alpha_{qlj} \|_{\psi_1}^2 &\leq \frac{1}{n} \sum_{k=1}^s \sum_{l=1}^{p-s} ([ \Sigma_{JJ^c}^{1/2}]_{il})^2 \bigg( (\Sigma^{1/2})_{JJ} U_J \bigg)_{kj}^2 \\
    &\leq  \frac{1}{n} \sum_{k=1}^s \bigg( (\Sigma^{1/2})_{JJ} U_J \bigg)_{kj}^2 \| \Sigma_{JJ^c}^{1/2} \|_{2\to\infty}^2 \\
    &\leq \frac{\lambda_1 \| \Sigma_{JJ^c}^{1/2} \|_{2\to\infty}^2 }{n}.
\end{align*}
Also,
\begin{align*}
    \max_{q,l,j} \| \alpha_{qlj} \|_{\psi_1} &\leq \frac{1}{n} \sqrt{\lambda_1} \| \Sigma_{JJ^c}^{1/2} \|_{2\to\infty}.
\end{align*}
By the the Generalized Bernstein inequality (Theorem \ref{thm:genbernstein} in  Appendix \ref{sec:bernstein}),
\begin{align*}
    \p\bigg( |\eta_{ij}| > t \bigg) \leq 2 \exp\bigg( -cn \min\bigg[\frac{t^2}{\lambda_1\| \Sigma_{JJ^c}^{1/2} \|_{2\to\infty}^2 }, \frac{t}{ \sqrt{\lambda_1}\| \Sigma_{JJ^c}^{1/2} \|_{2\to\infty}}\bigg] \bigg).
\end{align*}
Again taking $t = C\| \Sigma_{JJ^c}^{1/2} \|_{2\to\infty} \sqrt{\lambda_1} \sqrt{\frac{\log(s) + \log(k) + 4\log(p)}{n}}$ shows that this holds with probability $1 - O(s\inv k\inv p^{-4})$.  Taking a union over all $s$ rows and $k$ columns of the matrix yields that
\begin{align*}
   \| \frac{1}{n} \Sigma_{JJ^c}^{1/2} Y_{J^c}\t Y_J  (\Sigma^{1/2})_{JJ} U_J\|_{\max} &\lesssim \| \Sigma_{JJ^c}^{1/2} \|_{2\to\infty}\sqrt{ \lambda_1} \sqrt{\frac{\log(p)}{n}} \\
    &\lesssim \lambda_1  \sqrt{\frac{\log(p)}{n}}.
\end{align*}
Applying precisely the same argument to the other term yields with probability $1 - O(p^{-4})$ that
\begin{align*}
 \|\frac{1}{n} \Sigma_{JJ^c}^{1/2} Y_{J^c}\t Y_J  (\Sigma^{1/2})_{JJ} U_{\perp} \|_{\max} &\lesssim \lambda_1 \sqrt{\frac{\log(p)}{n}}.
 \end{align*}
 Therefore, combining these bounds with the initial bound in \eqref{bigbaby} and Proposition \ref{lem:dk} and the equivalent expressions for the $\sin\Theta$ distances (Lemma \ref{lem:sintheta} in Appendix \ref{sec:bernstein}), we have that with probability at least $1 - O(p^{-4})$,
 \begin{align*}
 \|K_2\|_{2\to\infty} = 
 \|\frac{1}{n} \Sigma_{JJ^c}^{1/2} Y_{J^c}\t Y_J (\Sigma^{1/2})_{JJ} \tilde U_J \|_{2\to\infty}  &\leq \sqrt{k} \| \frac{1}{n} \Sigma_{JJ^c}^{1/2} Y_{J^c}\t Y_J  (\Sigma^{1/2})_{JJ} U_J\|_{\max} \\
    &\qquad + \sqrt{s} \|\frac{1}{n} \Sigma_{JJ^c}^{1/2} Y_{J^c}\t Y_J  (\Sigma^{1/2})_{JJ} U_{\perp} \|_{\max}  \| U_{\perp}\t \tilde U_J\| \\
    &\lesssim \lambda_1 \sqrt{\frac{k\log(p)}{n}} + \frac{\lambda_1^2}{\lambda_k - \lambda_{k+1}} \frac{s\log(p)}{n}
   \end{align*}
 as desired. 
\end{proof}

\subsection{Proof of Lemma \ref{lem:K3}} \label{sec:proofsk3k4}
It will be useful to collect some properties of the matrix $\Sigma_{JJ^c}^{1/2}$, which we state as a proposition. 

\begin{proposition}[Properties of the Matrix $\Sigma_{JJ^c}^{1/2}$] \label{prop:sigma_props}
The matrix $\Sigma_{JJ^c}^{1/2}$ satisfies
\begin{align*}
    \| \Sigma_{JJ^c}^{1/2} \| 
           &\leq 2\sqrt{\lambda_1}
\end{align*}
Furthermore, the left singular subspace of $\Sigma^{1/2}_{JJ^c}$ must contain columns of $U_{\perp}$.
\end{proposition}
\begin{proof}[Proof of Proposition \ref{prop:sigma_props}]
First, we note that
\begin{align*}
    \| \Sigma_{JJ^c}^{1/2} \| 
   &= \bigg \| \begin{pmatrix} 0 & (\Sigma^{1/2})_{JJ^c}\\ ((\Sigma^{1/2})_{JJ^c})\t & 0 \end{pmatrix}\bigg\| \\
        &\leq \| \Sigma \|^{1/2} + \bigg\| \begin{pmatrix} (\Sigma^{1/2})_{JJ} & 0 \\ 0 & \Sigma_{J^c J^c}^{1/2} \end{pmatrix} \bigg\| \\
        &\leq 2\sqrt{\lambda_1},
\end{align*}
since eigenvalues bound eigenvalues of any principal submatrix. For the second claim, note that 
\begin{align*}
    \Sigma^{1/2} \begin{pmatrix}U_J \\0 \end{pmatrix} &= \begin{pmatrix}
    (\Sigma^{1/2})_{JJ}& (\Sigma^{1/2})_{JJ^c}\\ ((\Sigma^{1/2})_{JJ^c})\t &  \Sigma_{J^c J^c}^{1/2} \end{pmatrix}  \begin{pmatrix}U_J \\0 \end{pmatrix} \\
    &= \begin{pmatrix} U_J \\ 0 \end{pmatrix} \Lambda^{1/2}.
\end{align*}
This shows that the matrix $(\Sigma^{1/2}_{JJ^c})\t$ satisfies $(\Sigma^{1/2}_{JJ^c})\t U_J = 0$, so that its null space must contain the space spanned by $U_J$.  However, this also shows that since $(\Sigma^{1/2}_{JJ^c})\t \in \R^{(p-s)\times s}$, then its rank is at most $s -k$.  Hence, define $(\Sigma^{1/2}_{JJ^c})\t= V_1 D V_2\t$ as the reduced singular value decomposition of  $(\Sigma^{1/2}_{JJ^c})\t$.  Since its rank is at most $s- k$, we have that $V_1 \in \mathbb{O}(p-s,s-k)$, $V_2 \in \mathbb{O}(s,s-k)$, and $D$ is an $s -k \times s-k$ diagonal matrix of singular values.  

Since  $(\Sigma^{1/2}_{JJ^c})\t U_J = V_1 D V_2\t U_J =0$, the term $V_2 \in \mathbb{O}(s,s-k)$ must span a space perpendicular to $U_J$.  The only matrix up to choice of basis in $\mathbb{O}(s,s-k)$ satisfying $V_2 \t U_J = 0$ is the matrix $U_{\perp}$, which establishes the second claim.
\end{proof}

Therefore, all this shows that
\begin{itemize}
    \item The left singular subspace of $\Sigma^{1/2}_{JJ^c}$ contains columns of $U_{\perp}$;
    \item Its singular values are all uniformly bounded by $2\sqrt{\lambda_1}$.
\end{itemize}

We are now prepared to prove Lemma \ref{lem:K3}. 
\kthree*


\begin{proof}[Proof of Lemma \ref{lem:K3}]
Let $\Sigma_{JJ^c}^{1/2}$ have singular value decomposition $ U_{\perp}D V\t$, where $U_{\perp} \in \mathbb{O}(s,s-k)$, $D_{ii} \geq 0$, $1\leq i\leq s- k$, $V \in \mathbb{O}(p-s, s - k)$.  We will show the result for $D_{ii} > 0$, though the same proof goes through if $D_{ii} = 0$ for some $i$.  

Then the term $K_3$ satisfies
\begin{align}
    \| K_3 \|_{2\to\infty} &= \| (\Sigma^{1/2})_{JJ} \frac{Y_J\t Y_{J^c}}{n} (\Sigma_{JJ^c}^{1/2})\t \tilde U_J \|_{2\to\infty} \nonumber \\
    &\leq \| (\Sigma^{1/2})_{JJ} \frac{Y_J\t Y_{J^c}}{n}  V D U_{\perp}\t \tilde U_J\|_{2\to\infty} \nonumber \\
     &\leq \| (\Sigma^{1/2})_{JJ} \frac{Y_J\t Y_{J^c}}{n}  V D U_{\perp}\t U_{\perp}\|_{2\to\infty} \| U_{\perp}\t \tilde U_J\| \nonumber \\
     &\leq \| (\Sigma^{1/2})_{JJ} \frac{Y_J\t Y_{J^c}}{n}  V  \|_{2\to\infty} \sqrt{\lambda_1} \| U_{\perp}\t \tilde U_J\|.  \label{K31}
\end{align}
The term $\|U_{\perp}\t \tilde U_J\|$ can be bounded via Proposition \ref{lem:dk} and Lemma \ref{lem:sintheta} in Appendix \ref{sec:bernstein}.  So what remains is to bound the $2\to\infty$ norm in \eqref{K31}.  Note that the matrix $V$ is of column dimension at most $(s-k)$. Hence, each of the $s$ rows of the matrix
$(\Sigma^{1/2})_{JJ} Y_J Y_{J^c} V$ is of dimension at most $s- k$.  

Following  a strategy similar to that in Lemmas \ref{lem:K1} and \ref{lem:K2}, we have that
\begin{align*}
    \| (\Sigma^{1/2})_{JJ} \frac{Y_J\t Y_{J^c}}{n}  V \|_{2\to\infty} &\leq \sqrt{s-k} \max_{i,j} \bigg| (\Sigma^{1/2})_{JJ} \frac{Y_J\t Y_{J^c}}{n}  V  \bigg|_{i,j} \\
    &\leq \sqrt{s} \max_{i,j} \bigg| (\Sigma^{1/2})_{JJ} \frac{Y_J\t Y_{J^c}}{n}  V  \bigg|_{i,j}.
\end{align*}
By analogous arguments as in Lemma  \ref{lem:K2}, the $i,j$ entry is a sum of independent mean-zero subexponential random variables, each with $\psi_1$ norm bounded $\frac{1}{n} \sqrt{\lambda_1}$.  Therefore, by Bernstein's inequality, any $i,j$ entry is bounded by
\begin{align*}
    C \sqrt{\lambda_1} \sqrt{\frac{\log(p)}{n}}
\end{align*}
with probability at most $1 - O(p^{-3})$.  Combining with Proposition \ref{lem:dk}, we have the bound
\begin{align*}
     \| K_3 \|_{2\to\infty} &\lesssim \lambda_1 \sqrt{\frac{s\log(p)}{n}} \| U_{\perp}\t \tilde U_J \| \\
     &\lesssim \lambda_1 \sqrt{\frac{s \log(p)}{n}} \bigg(  \frac{\lambda_1}{\lambda_k- \lambda_{k+1}}\sqrt{\frac{s \log(p)}{n}}\bigg)\\
     &\lesssim \frac{\lambda_1^2 }{\lambda_k - \lambda_{k+1}} \frac{ s  \log(p)}{n}
\end{align*}
as desired.

For the term $K_4$, we see that 
\begin{align}
    \| K_4 \|_{2\to\infty} &= \| \Sigma^{1/2}_{JJ^c} \bigg( \frac{Y_{J^c}\t Y_{J^c}}{n} - I \bigg) VD U_{\perp}\t \tilde U_J\|_{2\to\infty} \nonumber \\
    &\leq  \| \Sigma^{1/2}_{JJ^c} \bigg( \frac{Y_{J^c}\t Y_{J^c}}{n} - I \bigg) V\|_{2\to\infty} \sqrt{\lambda_1} \| U_{\perp}\t \tilde U_J\| \nonumber \\
    &\leq \sqrt{s\lambda_1} \| \Sigma^{1/2}_{JJ^c} \bigg( \frac{Y_{J^c}\t Y_{J^c}}{n} - I \bigg) V\|_{\max}  \| U_{\perp}\t \tilde U_J\|. \label{bigboybigboy}
\end{align}
We will bound the term inside the max norm for fixed $i$ and $j$.   Observe that 
\begin{align*}
   \bigg| \bigg( \Sigma^{1/2}_{JJ^c} \bigg( \frac{Y_{J^c}\t Y_{J^c}}{n} - I \bigg) V \bigg)_{ij} \bigg| &= \max_{i,j} \bigg| \frac{1}{n} \sum_{q} \sum_{k,l} \bigg( \Sigma^{1/2}_{JJ^c}\bigg)_{ik} ( Y_{qk} Y_{ql} - \E Y_{qk} Y_{ql} ) V_{lj}\bigg| \\
  &\leq  \max_q \bigg| \sum_{k,l} \bigg( \Sigma^{1/2}_{JJ^c}\bigg)_{ik} ( Y_{qk} Y_{ql} - \E Y_{qk} Y_{ql} ) V_{lj}\bigg|.
\end{align*}
We will first bound the term inside the absolute value for fixed $q$ by Hanson-Wright (Theorem \ref{thm:hw} in Appendix \ref{sec:bernstein}).  Let $M$ be the matrix defined via
\begin{align*}
    M_{kl} := \bigg( \Sigma_{JJ^c}^{1/2} \bigg)_{ik} V_{lj}.
\end{align*}
Then
\begin{align*}
    \| M \|_F^2 &= \sum_{k,l} \bigg( \Sigma_{JJ^c}^{1/2} \bigg)_{ik}^2 V_{lj}^2 
    = \sum_k  \bigg( \Sigma_{JJ^c}^{1/2} \bigg)_{ik}^2 
    \leq \| \Sigma_{JJ^c}^{1/2} \|_{2\to\infty}^2 
    \leq 4\lambda_1.
\end{align*}
Therefore, by applying the Hanson-Wright inequality, for any fixed $q$ it holds that
\begin{align*}
  \p\bigg(  \bigg| \sum_{k,l} \bigg( \Sigma^{1/2}_{JJ^c}\bigg)_{ik} ( Y_{qk} Y_{ql} - \E Y_{qk} Y_{ql} ) V_{lj}\bigg| \geq t\bigg) &\leq 2 \exp\bigg( -c \min\bigg\{ \frac{t^2}{4\lambda_1 }, \frac{t}{\|M\|} \bigg\} \bigg).
\end{align*}
Setting $t = C \sqrt{\lambda_1} \sqrt{\frac{\log(s) + \log(k) + 5\log(p)}{n}}$ and taking a union bound over all $q$ random variables shows that for fixed $i$ and $j$, with probability at least $1 - O(s\inv k\inv p^{-4})$,
\begin{align*}
     \bigg| \bigg( \Sigma^{1/2}_{JJ^c} \bigg( \frac{Y_{J^c}\t Y_{J^c}}{n} - I \bigg) V \bigg)_{ij} \bigg| &\lesssim\sqrt{\lambda_1} \sqrt{\frac{\log(p)}{n}}.
\end{align*}
Taking a union bound over $s$ rows and  $k$ columns shows that with probability at least $1 - O(p^{-4})$,
\begin{align*}
    \| \Sigma^{1/2}_{JJ^c} \bigg( \frac{Y_{J^c}\t Y_{J^c}}{n} - I \bigg) V\|_{\max} &\lesssim \sqrt{\lambda_1} \sqrt{\frac{\log(p)}{n}}.
\end{align*}
Therefore, from the initial bound in \eqref{bigboybigboy} and Proposition \ref{lem:dk}, 
\begin{align*}
    \| K_4 \|_{2\to\infty} &= \| \Sigma^{1/2}_{JJ^c} \bigg( \frac{Y_{J^c}\t Y_{J^c}}{n} - I \bigg) VD U_{\perp}\t \tilde U_J\|_{2\to\infty} \\
    &\leq \sqrt{s\lambda_1} \| \Sigma^{1/2}_{JJ^c} \bigg( \frac{Y_{J^c}\t Y_{J^c}}{n} - I \bigg) V\|_{\max}  \| U_{\perp}\t \tilde U_J\| \\
    &\lesssim  \lambda_1  \sqrt{\frac{s\log(p)}{n}}\| U_{\perp}\t \tilde U_J\| \\
    &\lesssim \frac{s \log(p)}{n} \frac{ \lambda_1^2 }{\lambda_k - \lambda_{k+1}}
\end{align*}
as desired. 
\end{proof}


\section{Background Material on Orlicz Norms, Concentration, and Subspace Perturbation}\label{sec:bernstein}
Here we briefly discuss Orlicz $\psi_\alpha$ Norms and Bernstein's inequality for subexponential random variables.

The \emph{Orlicz Norm} of order $\alpha$ for a real-valued random variable $X$ is defined via
\begin{align*}
\| X \|_{\psi_{\alpha}} :=    \inf \{ t > 0: \E \exp(|X|^{\alpha}/t) \leq 1 \}.
\end{align*}
Random variables with finite $\psi_{2}$ norm are called \emph{subgaussian} and those with a finite $\psi_1$ norm are called \emph{subexponential}.  Generally speaking, if $X$ is subgaussian, then $X^2$ is subexponential and $\| X^2\|_{\psi_1} \lesssim \| X\|_{\psi_2}^2$.  One also has the ``Cauchy-Schwarz'' bound $\| XY\|_{\psi_1} \lesssim \| X\|_{\psi_2} \|Y\|_{\psi_2}$ \citep{vershynin_high-dimensional_2018}. 

For subexponential random variables, one has the following generalized Bernstein's inequality.  See Theorem 2.8.2 in \citet{vershynin_high-dimensional_2018} for the proof.
\begin{theorem}[Theorem 2.8.2 in \citet{vershynin_high-dimensional_2018}] \label{thm:genbernstein}
Let $X_1, ..., X_N$ be independent, mean zero subexponential random variables and let $a = (a_i)_{i=1}^{N}$.  Then there exists a universal constant $c >0$ such that for all $t \geq 0$, we have that
\begin{align*}
   \mathbb{P}\left\{\left|\sum_{i=1}^{N} a_{i} X_{i}\right| \geq t\right\} \leq 2 \exp \left[-c \min \left(\frac{t^{2}}{K^{2}\|a\|_{2}^{2}}, \frac{t}{K\|a\|_{\infty}}\right)\right]
\end{align*}
where $K = \max_{i} \| X_i\|_{\psi_1}$.
\end{theorem}

We also make use of the Hanson-Wright Inequality.  See Theorem 6.2.1 in \citet{vershynin_high-dimensional_2018} for the proof.

\begin{theorem}[Hanson-Wright Inequality --Theorem 6.2.1 in \citet{vershynin_high-dimensional_2018}] \label{thm:hw}
Let $X_1, \dots, X_N$ be independent, mean-zero  subgaussian random variables.  Let $M$ be some fixed $N \times N$ matrix.  Then there exists a universal constant $c >0$ such that for all $t \geq 0$, we have that
\begin{align*}
    \p\bigg\{ \bigg| \sum_{k,l} M_{kl} X_{k} X_{l}  - \E M_{kl} X_{k} X_{l}  \bigg|  \geq t \bigg\} &\leq 2 \exp\bigg( -c \min\bigg\{ \frac{t^2}{K^4 \|M\|_F^2}, \frac{t}{K^2 \|M\|} \bigg\} \bigg),
\end{align*}
where $K = \max_i \| X_i \|_{\psi_2}$.
\end{theorem}

We also use several notions from subspace perturbation theory.  Suppose $U$ and $\hat U$ are two $d_1 \times d_2$ matrices with orthonormal columns with $d_2 \leq d_1$.  The $\sin\Theta$ distance between the subspaces spanned by $U$ and $\hat U$ is defined as follows.  Let $I - UU\t = U_{\perp} U_{\perp}\t$.  Then the (spectral) $\sin\Theta$ distance is defined as
\begin{align*}
    \| \sin\Theta(U_1, U_2) \| :&= \| \hat U\t U_{\perp}\|.
\end{align*}
Throughout the supplementary material, we use several equivalent terms for the $\sin\Theta$ distance.  We present this here as a lemma, the statement of which is slightly modified from Lemma 1 of \citet{cai_rate-optimal_2018}.  

\begin{lemma}[Modified from Lemma 1 of \citet{cai_rate-optimal_2018}] \label{lem:sintheta}
The $\sin\Theta$ distance between two matrices satisfies
\begin{align*}
    \| \sin\Theta(\hat U, U)\| \leq \inf_{W: WW\t = I_{d_2}} \| \hat U - U W \| \leq  \sqrt{2} \| \sin\Theta(\hat U, U)\|; \\
    \| \sin\Theta(\hat U, U) \| \leq \| \hat U \hat U\t - UU\t \| \leq 2 \|\sin\Theta(\hat U, U) \|.\end{align*}
\end{lemma}




\end{document}